\documentclass[11pt,reqno,a4paper]{amsart}
\usepackage{hyperref}
\usepackage{latexsym}
\usepackage{amssymb,amsmath}
\usepackage{graphics}
\usepackage{url}
\usepackage{bbm}
\usepackage{multirow}
\usepackage[vcentermath]{youngtab}
  
\usepackage{booktabs}  
\usepackage{enumerate}
\usepackage[usenames,dvipsnames]{color}
\usepackage{dsfont}
\allowdisplaybreaks

\usepackage{braket} 
\usepackage{multirow}

\usepackage{tikz}
\usetikzlibrary{matrix,chains,shapes,arrows,scopes,positioning}
\usetikzlibrary{calc,decorations.pathmorphing,shapes}
\usetikzlibrary{patterns}
\tikzstyle{empty}=[circle,draw=black!80,thick]
\tikzstyle{emptyn}=[circle,draw=black!80,fill=white,scale=0.5] 
\tikzstyle{nero}=[circle,draw=black!80,fill=black!80,thick] 

\newcommand{\fS}{\mathfrak{S}}

\newcommand{\Irr}{{\mathrm{Irr}}}

\newcommand{\Lin}{\mathrm{Lin}}

\newcommand{\down}{\big\downarrow}
\newcommand{\up}{\big\uparrow}

\newcommand{\cd}{\mathrm{cd}}
\newcommand{\lX}{\mathcal{X}}

\newcommand{\bB}{\bar{\mathcal{B}}}

\newtheorem{theorem}{Theorem}[section]
\newtheorem{lemma}[theorem]{Lemma}

\newtheorem{proposition}[theorem]{Proposition}

\theoremstyle{definition}
\newtheorem{example}[theorem]{Example}

\newtheorem{definition}[theorem]{Definition}

\newtheorem{observation}[theorem]{Remark}

\newcommand{\6}{^}

\begin{document}

\title[]{Sylow Branching Coefficients and Hook partitions}

%\date{\today}

\author{Eugenio Giannelli}
\address[E. Giannelli]{Dipartimento di Matematica e Informatica U.~Dini, Viale Morgagni 67/a, Firenze, Italy}
%\email{eugenio.giannelli@unifi.it}

\author{Giada Volpato}
\address[G. Volpato]{Dipartimento di Matematica e Informatica U.~Dini, Viale Morgagni 67/a, Firenze, Italy}
\email{eugenio.giannelli@unifi.it, giada.volpato@unifi.it}

%\thanks{The first author's research was funded by Trinity Hall, University of Cambridge.}
%INDAM?

\begin{abstract}
We give a description of the irreducible constituents of the restriction to Sylow $2$-subgroups of irreducible characters of symmetric groups labelled by hook partitions. 
\end{abstract}


%\subjclass[]{}

\maketitle

%%%%%%%%%%%%%%%%%%%%%%%%%%%%%%%%%%%%%%%%%%%%%%%%%%%%%%%%%%%%%%%%%%%%%%%%%
\section{Introduction}\label{sec:intro}

Let $n\in\mathbb{N}$ and let $\fS_n$ be the finite symmetric group of degree $n$. Let $p$ be a prime number and let $P_n$ be a Sylow $p$-subgroup of $\fS_n$. Let $\chi\in\mathrm{Irr}(\fS_n)$ and $\phi\in\mathrm{Irr}(P_n)$ be irreducible characters of $\fS_n$ and $P_n$, respectively. The corresponding \textit{Sylow branching coefficient} $Z_{\phi}\6{\chi}$ is defined as the multiplicity of $\phi$ as an irreducible constituent of $\chi\big\downarrow_{P_n}$, the restriction of $\chi$ to $P_n$. Sylow branching coefficients have been intensely studied for symmetric groups and odd primes \cite{INOT, GL1, GL2}. Moreover, new information on these integers was recently used in \cite{GLLV} to prove a conjecture proposed by Malle and Navarro in \cite{MN12}. 
There is now an evident fracture between the knowledge accumulated on Sylow branching coefficients at odd primes and the lack of information on this topic when the prime $p$ is equal to $2$. For instance
the irreducible constituents of the Sylow permutation character $1_{P_n}\big\uparrow\6{\fS_n}$
are completely described for odd primes \cite[Theorem A]{GL1} but, despite some recent advances \cite{LO}, are far from understood when $p=2$. 
 
The aim of this note is to advance in the study of Sylow branching coefficients at the prime $2$. In particular, here we focus on irreducible characters labelled by \textit{hook partitions}. From now on we denote this subset of $\mathrm{Irr}(\fS_n)$ by $\mathrm{Irr}_{\mathcal{H}}(\fS_n)$.
In order to present our main results, we first recall that $P_{2^k}$ is isomorphic to the $k$-fold wreath product of cyclic groups of order $2$. It follows that linear characters of $P_{2^k}$ are naturally labelled by $\{0,1\}$-sequences of length $k$ (see the end of Section \ref{sec: wr}).  
Motivated by the study of the McKay Conjecture, in \cite{GJLMS} it was shown that for every $\chi\in\mathrm{Irr}_{\mathcal{H}}(\fS_{2^k})$,
the restriction of $\chi$ to $P_{2^k}$ admits a unique linear constituent. In \textbf{Theorem \ref{lem: unique_linear}} below we identify this unique constituent by describing the associated $\{0,1\}$-sequence. This result is immediately used in Section \ref{S: 4} to compute a large family of Sylow branching coefficients. In particular, in \textbf{Theorem \ref{thm: 4.3}} we calculate $Z^\chi_\phi$ for all $\chi\in\mathrm{Irr}_{\mathcal{H}}(\fS_n)$ and all $\phi\in\mathrm{Lin}(P_n)$. This is a wide generalization of \cite[Theorem 1.1]{GJLMS} and, at the moment, one of the few examples of explicit numerical determination of Sylow branching coefficients for arbitrary large symmetric groups. 
In the second part of the article we move from linear to arbitrary constituents. In particular, given $n, k\in\mathbb{N}$ we investigate which $\chi\in\mathrm{Irr}_{\mathcal{H}}(\fS_n)$ are such that $\chi\big\downarrow_{P_n}$ admits an irreducible constituent of degree $2^k$. In \textbf{Theorem \ref{thm: generic_case}} we address this question. An interesting consequence of our observations in Section \ref{S: 5} is that whenever $\chi\big\downarrow_{P_n}$ admits a constituent of degree $2\6k$, then it necessarily admits constituents of degree $2^\ell$, for every $0\leq \ell\leq k$. 
%In particular, given a sequence $(x_1, \ldots, x_k)\in \{0,1\}\6{\times k}$ we denote by $\mathcal{X}(x_1, \ldots, x_k)$ the corresponding linear character of $P_{2^k}$. 

\section{Notation and background}

Given integers $n\leq m$, we denote by $[n,m]$ the set $\{n, n+1,\ldots, m\}$. If $n<m$ then $[m,n]$ is regarded as the empty set. 
We let $\mathcal{C}(n)$ be the set of all compositions of $n$, i.e. the set consisting of all the finite sequences $(a_1,a_2,\ldots, a_z)$ such that $a_i$ is a non-negative integer for all $i\in [1,z]$ and such that $a_1+\cdots+a_z=n$.
Given $\lambda=(\lambda_1,\dots , \lambda_z)\in \mathcal{C}(n)$, we sometimes denote by $l(\lambda)=z$ the number of parts of $\lambda$.
As usual, given a finite group $G$, we denote by $\mathrm{Irr}(G)$ the set of irreducible complex characters of $G$, and by $\mathrm{Lin}(G)$ the subset of linear characters of $G$. Finally, $\mathrm{cd}(G)=\{\chi(1)\ |\ \chi\in\mathrm{Irr}(G)\}$ is the set of irreducible character degrees of $G$. 

\subsection{Wreath products and Sylow Subgroups}\label{sec: wr}
Here we fix the notation for characters of wreath products. For more details see \cite[Chapter 4]{JK}.
Let $G$ be a finite group and let $H$ be a subgroup of $\mathfrak{S}_n$. We denote by $G^{\times n}$ the direct product of $n$ copies of $G$. The natural action of $\mathfrak{S}_n$ on the direct factors of $G^{\times n}$ induces an action of $\mathfrak{S}_n$ (and therefore of $H\leq \mathfrak{S}_n$) via automorphisms of $G^{\times n}$, giving the wreath product $G\wr H := G^{\times n} \rtimes H$. We refer to $G^{\times n}$ as the base group of the wreath product $G\wr H$. 
We denote the elements of $G\wr H$ by $(g_1,\dots , g_n ;h)$ for $g_i\in G$ and $h\in H$. Let $V$ be a $\mathbb{C}G$-module and suppose it affords the character $\phi$.
We let $V^{\otimes n} := V \otimes \cdots \otimes V$ ($n$ copies) be the corresponding $\mathbb{C} G^{\times n}$-module. The left action of $G\wr H$ on $V^{\otimes n}$ defined by linearly extending $$(g_1,\dots ,g_n; h) : v_1 \otimes \cdots \otimes v_n \mapsto g_1 v_{h^{-1}(1)} \otimes \cdots \otimes g_n v_{h^{-1}(n)},$$
turns $V^{\otimes n}$ into a $\mathbb{C}(G\wr H)$-module, which we denote by $\tilde{V}^{\otimes n}$.
We denote by $\tilde{\phi}$ the character afforded by the $\mathbb{C}(G\wr H)$-module $\tilde{V}^{\otimes n}$.
For any character $\psi$ of $H$, we let $\psi$ also denote its inflation to $G\wr H$ and let 
$\mathcal{X} (\phi; \psi):= \tilde{\phi} \cdot \psi$
be the character of $G\wr H$ obtained as the product of $\tilde{\phi}$ and $\psi$.
Let $\phi \in {\rm Irr}(G)$ and let $\phi^{\times n}:= \phi \times \cdots \times \phi$ be the corresponding irreducible character of $G^{\times n}$. Observe that $\tilde{\phi}\in {\rm Irr}(G\wr H)$ is an extension of $\phi^{\times n}$.
Given $K\leq G$, we denote by ${\rm Irr}(G|\psi)$ the set of characters $\chi \in {\rm Irr}(G)$ such that $\psi$ is an irreducible constituent of the restriction $\chi \down_K$. Hence, by Gallagher's Theorem \cite[Corollary 6.17]{IBook} we have
\[{\rm Irr}(G\wr H | \phi^{\times n})= \Set{ \mathcal{X} (\phi; \psi) | \psi \in {\rm Irr}(H)}. \]

If $H=C_p$ is a cyclic group of prime order $p$, every $\psi \in \Irr(G\wr C_p)$ is either of the form
\begin{itemize}
	\item[(i)] $\psi= \phi_1 \times \cdots \times \phi_p \up^{G\wr C_p}_{G^{\times p}}$, where $\phi_1 , \dots \phi_p \in \Irr(G)$ are not all equal; or
	\item[(ii)] $\psi= \lX (\phi; \theta)$ for some $\phi \in \Irr(G)$ and $\theta \in \Irr(C_p)$. 
\end{itemize}
We remark that in case (i) we have that ${\rm Irr}(G\wr C_p | \phi_1 \times \cdots \times \phi_p)=\{\psi\}$.

Finally, we record some facts about Sylow subgroups of symmetric group and we refer to \cite[Chapter 4]{JK} or to \cite{Olsson} for more details.
Fix $p$ a prime number. We let $P_n$ denote a Sylow $p$-subgroup of $\fS_n$. Clearly $P_1$ is the trivial group while $P_p\cong C_p$ is cyclic of order $p$. If $i\geq 2$, then $P_{p^i}=\big(P_{p^{i-1}}\big)^{\times p} \rtimes P_p=P_{p^{i-1}}\wr P_p\cong P_p\wr \cdots \wr P_p$ ($i$-fold wreath product). Let $n=\sum_{i=1}^r p^{n_i}$, with $n_1 \geq \dots \geq n_r \geq 0$, be the $p$-adic expansion of $n$. Then $P_n \cong P_{p^{n_1}} \times \cdots \times P_{p^{n_r}}$.
When $p=2$, we let $\{\phi_0, \phi_1\}=\mathrm{Irr}(P_2)$, where $\phi_0$ denotes the trivial character of $P_2$. Using the facts on representations of wreath products highlighted above, it is easy to observe that linear characters of $P_{2^n}$ are naturally labelled by elements of $\{0,1\}^{\times n}$. In fact, setting $\mathcal{X}(0)=\phi_0$, $\mathcal{X}(1)=\phi_1$ and given $(i_1,\ldots, i_n)\in \{0,1\}^{\times n}$, we recursively define $\mathcal{X}(i_1, \dots, i_{n-1}, i_{n})\in \mathrm{Lin}(P_{2^n})$ as $\mathcal{X}(i_1, \dots, i_{n-1}, i_{n})=\mathcal{X}(\mathcal{X}(i_1, \dots , i_{n-1}); \phi_{i_n}).$

%For $n\in \mathbb{N}$, the normalizer of a Sylow $p$-subgroup of $\fS_{p^n}$ is $N_{\fS_{p^n}}(P_{p^n}) = P_{p^n} \rtimes H$, where $H\cong (C_{p-1})^{\times n}$.
%More generally, if $n=\sum_{i=1}^t a_i p^{n_i}$, $n_t > \cdots > n_1 \geq 0$, then $N_{\fS_n}(P_n)=N_1 \wr \fS_{a_1} \times \cdots \times N_t \wr \fS_{a_t}$, where $N_i:=N_{\fS_{p^{n_i}}}(P_{p^{n_i}})$ for every $i\in [1,t]$.
%As explained in \cite[Proof of Lemma 4.2]{Olsson}, we have that $N_{\fS_{p^n}}(P_{p^n})=N_{\fS_{p^{n-1}}\wr \fS_p}(P_{p^n}) \leq \fS_{p^{n-1}}\wr \fS_p $. Thus the elements of elements of $N_{\fS_{p^n}}(P_{p^n})$ can be written as $(x_1,\dots ,x_p ;y)$, with $x_i \in \fS_{p^{n-1}}$ for all $i \in [1,p] $ and $y\in \fS_p$.
%We refer the reader to \cite[Section 2]{G21} for more details about the structure of the normaliser of a Sylow $p$-subgroup. 

\subsection{The Littlewood-Richardson coefficients}
For each $n\in \mathbb{N}$, $\Irr(\fS_n)$ is naturally in bijection with $\mathcal{P}(n)$, the set of all partitions of $n$. For $\lambda \in \mathcal{P}(n)$, the corresponding irreducible character is denoted by $\chi^\lambda$.
Let $m,n\in \mathbb{N}$ with $m<n$. Given $\chi\6\mu\times\chi\6\nu\in\mathrm{Irr}(\fS_m\times\fS_{n-m})$, the decomposition into irreducible constituents of the induction
\[\left( \chi^\mu \times \chi^\nu \right) \up^{\mathfrak{S}_n}= \sum_{\lambda\in \mathcal{P}(n)} \mathcal{LR}(\lambda;\mu ,\nu) \chi^{\lambda}\]
is described by the Littlewood-Richardson rule (see \cite[Chapter 5]{Fulton} or \cite[Chapter 16]{J}).
Here the natural numbers $\mathcal{LR}(\lambda;\mu ,\nu)$ are called \emph{Littlewood-Richardson coefficients}.

Given $(n_1,\ldots, n_k) \in\mathcal{C}(n)$, $\lambda\in\mathcal{P}(n)$ and $\mu_j\in\mathcal{P}(n_j)$ for all $j\in [1,k]$, we let $\mathcal{LR}(\lambda; \mu_1,\ldots, \mu_k)$ be the multiplicity of $\chi^\lambda$ as an irreducible constituent of $(\chi^{\mu_1}\times\cdots\times\chi^{\mu_k}) \up^{\mathfrak{S}_n}_{Y}$. 
Here $Y$ denotes the Young subgroup $\fS_{n_1}\times \fS_{n_2}\times\cdots\times \fS_{n_k}$ of $\fS_n$. Later in the article, we will sometimes alternatively denote the subgroup $Y$ by $\fS_{(n_1,\ldots, n_k)}$.
The following lemma describes the behavior of the first parts of the partitions involved in a non-zero Littlewood-Richardson coefficient. This will be useful in the following sections. 

\begin{lemma}\label{lem: LR_prop}
	If $\mathcal{LR}(\lambda; \mu_1, \dots , \mu_k)\neq 0$ then $\lambda_1 \leq \sum_{j=1}^k (\mu_j)_1$.
\end{lemma}
\begin{proof}
	When $k=2$, the statement is a straightforward consequence of the combinatorial description of the Littlewood-Richardson coefficient 
	$\mathcal{LR}(\lambda; \mu_1, \mu_2)$, as given in \cite[Section 5.2]{Fulton}. The lemma is then proved by iteration. 
\end{proof}

As in \cite{GL2}, we define $\mathcal{B}_n(t)$ as the set of those partitions of $n$ whose Young diagram fits inside a $t\times t$ square grid, i.e. for $n,t \in \mathbb{N}$, we set
\[ \mathcal{B}_n(t):=\Set{\lambda \in \mathcal{P}(n) | \lambda_1 \leq t, \ l(\lambda)\leq t}. \]
%Notice that $\mathcal{B}_n(t)$ is closed under taking conjugate partitions.

%The operator $\star$ 
Moreover, for $(n_1,\ldots, n_k)\in\mathcal{C}(n)$ and $A_j\subseteq \mathcal{P}(n_j)$ for all $j\in [1, k]$, we let \[A_1\star A_2\star\cdots\star A_k := \Set{\lambda \in \mathcal{P}(n) | \mathcal{LR}(\lambda; \mu_1,\ldots, \mu_k) > 0,\ \text{for some}\ \mu_1\in A_1, \ldots, \mu_k\in A_k}.\]
It is easy to check that $\star$ is both commutative and associative.
The following lemma was first proved in \cite[Proposition 3.3]{GL2}.
\begin{lemma}\label{lem: B star}
	Let $n,n',t,t'\in\mathbb{N}$ be such that $\tfrac{n}{2}<t\le n$ and $\tfrac{n'}{2}<t'\le n'$. Then 
	$$\mathcal{B}_n(t) \star\mathcal{B}_{n'}(t') = \mathcal{B}_{n+n'}(t+t').$$
\end{lemma}

\section{On the McKay bijection for symmetric groups}\label{sec: 3}

From now on we fix $p=2$ and we let $P_n$ be a Sylow $2$-subgroup of $\fS_n$. Motivated by the study of McKay bijections for symmetric groups, 
in \cite[Theorem 1.1]{GJLMS} it is shown that the restriction to $P_{2^n}$ of every irreducible character of odd degree of $\fS_{2^n}$ admits a unique linear constituent. In this section we will clarify which one by explicitly determining the corresponding $\{0,1\}$-sequence, as described in Section \ref{sec: wr}. 
%In this case the normalizer $N_{\fS_{2^n}}(P_{2^n})$ is equal to $P_{2^n}$ itself, and its irreducible characters with degree not divisible by $2$ are clearly the linear ones.
We recall that by \cite[Lemma 3.1]{GJLMS} we know that the irreducible characters of odd degree of $\fS_{2^n} $ are exactly those labelled by hook partitions of $2^n$.
In other words we have that $\mathrm{Irr}_{2'}(\fS_{2^n})=\mathrm{Irr}_{\mathcal{H}}(\fS_{2^n})$. We introduce the following last useful piece of notation. For a natural number $n$, we let $\mathcal{H}(n)$ be the set of hook partitions of the natural number $n$. To simplify the notation, we write $(n-x,1^x)$ for the hook partition $(n-x,1, \dots , 1)\in \mathcal{H}(n)$.
For the convenience of the reader we record here \cite[Theorem 1.1]{GJLMS}.
\begin{theorem}\label{thm: GMcKay}
Let $n\in\mathbb{N}$ and let $\lambda \in\mathcal{H}(2^n)$. Then $\chi^\lambda \down_{P_{2^n}}$ admits a unique linear constituent. Such a constituent appears with multiplicity $1$. 
\end{theorem}

From now on we will adopt the following notation. Given $z \in \mathbb{Z}$, we let $[z]\in \{0,1\}$ be such that $z\equiv\ [z]\ \mathrm{mod}\ 2$. 
We are now ready to state our first result. 

\begin{theorem}\label{lem: unique_linear}
Let $n\in \mathbb{N}$ and let $\lambda =(2^n -x, 1^x) \in \mathcal{H}(2^n)$ where $x=a_n 2^n + a_{n-1} 2^{n-1} + \cdots + a_0 2^0$ is its binary expansion.
The unique linear constituent of $\chi^\lambda \down_{P_{2^n}}$ is $$\mathcal{X}\left( [a_n+a_{n-1}], [a_{n-1}+a_{n-2}], \dots , [a_1+a_0]\right).$$
\end{theorem}

\begin{proof}
We proceed by induction on $n$. If $n=1$ then $\lambda \in \{(2),(1^2)\}$. If $\lambda=(2)$ then $x=0=0 \cdot 2^1+0 \cdot 2^0$, $[a_1+a_0]=[0+0]=0$ and $\chi^{(2)} \down_{P_2} = \phi_0$. Similarly, $\chi^{(1^2)} \down_{P_2} =\phi_1$.	
Let now $n\geq 2$ and $\lambda =\left(2^n -x, 1^x\right)$. We denote by $L_{\lambda}$ the unique linear constituent of $\chi^\lambda \down_{P_{2^n}}$, as prescribed by Theorem \ref{thm: GMcKay}. By the Littlewood-Richardson rule, we have
	\[\chi^\lambda \down_{\fS_{2^{n-1}} \times \fS_{2^{n-1}} } = \big(\chi^{\left(2^{n-1}-y,1^y\right)} \times \chi^{\left(2^{n-1}-y,1^y\right)}\big) + \Delta , \]
	where $y= \begin{cases} \frac{x}{2}, &\mbox{ if } x \mbox{ is even,} 
		\\ \frac{x-1}{2}, &\mbox{ if }x \mbox{ is odd},
	\end{cases}$ and $\Delta$ is the sum of irreducible constituents of the form $\phi\times \psi$, with $\phi\neq \psi$. 
It follows that $L_\lambda = \mathcal{X}\left(L_{\left(2^{n-1}-y,1^y\right)} ; \phi_\alpha\right)$, for some $\alpha\in \{0,1\}$. In order to find $\alpha$, we consider a $2^n$-cycle $g\in P_{2^n}$. Since $P_{2^n}=P_{2^{n-1}}\wr P_2$ we can choose $g=(h, 1;\gamma)$, where $h\in P_{2^{n-1}}$ is a  $2^{n-1}$-cycle and $\gamma\in P_2$ is a $2$-cycle. Using \cite[4.3.9]{JK} and \cite[Lemma 3.11]{GN}, we have that $$\chi^\lambda(g)=L_\lambda(g)=L_{\left(2^{n-1}-y,1^y\right)}(h)\cdot\phi_\alpha(\gamma)=\chi^{\left(2^{n-1}-y,1^y\right)}(h)\cdot \chi^{(2-\alpha, \alpha)}(\gamma).$$ The Murnaghan-Nakayama rule \cite[2.4.7]{JK} implies that 
$(-1)^x=(-1)^y(-1)^\alpha$. It follows that
	\[ \alpha = \begin{cases} 0, &\mbox{ if }x \equiv\ y\ \mathrm{mod}\ 2, 
		\\ 1, &\mbox{ otherwise}. \end{cases} \]
Since $x=a_n 2^n+ \cdots + a_0 2^0$ is the binary expansion of $x \in \left[0, 2^n -1\right]$, then $y=a_n 2^{n-1}+ \cdots + a_1 2^0$ is the binary expansion of $y\in \left[0, 2^{n-1}-1\right]$. 
	Hence, $\alpha= [a_1+a_0]$ and we have $L_\lambda = \mathcal{X}\left(L_{\left(2^{n-1}-y,1^y\right)}; \phi_{[a_1+a_0]}\right) $. Using the binary expansion of $y$ and the inductive hypothesis we conclude the proof.\end{proof}

\section{Computing Linear Sylow Branching Coefficients}\label{S: 4}

The aim of this section is to compute the Sylow branching coefficients $Z_{\phi}^\chi$ for all $\chi\in\mathrm{Irr}_{\mathcal{H}}(\fS_n)$ and all $\phi\in\mathrm{Lin}(P_n)$. This generalizes Theorem \ref{thm: GMcKay} to any arbitrary natural number.

Given $t\in\mathbb{N}$ and $y\in\mathbb{Z}$ we let ${t\choose y}$ be the usual binomial coefficient. In particular this is regarded as $0$ whenever $y\notin [0,t]$.

\begin{lemma}\label{lem: hook-hooks}
Let $\alpha=(n_1,\ldots, n_t)\in\mathcal{C}(n)$ and $h\in \mathcal{H}(n)$. Let $\lambda_i\in\mathcal{P}(n_i)$ for all $i\in [1,t]$. 
If $\mathcal{LR}(h; \lambda_1,\ldots, \lambda_t)\neq 0$ then $\lambda_i\in\mathcal{H}(n_i)$ for all $i\in [1,t]$. 
\end{lemma}
\begin{proof}
This is a direct consequence of the Littlewood-Richardson rule.
\end{proof}

\begin{lemma}\label{lem: diff}
Let $\alpha=(n_1,\ldots, n_t)\in\mathcal{C}(n)$ and $h=(n-x,1^x)\in \mathcal{H}(n)$. 
For each $i\in [1,t]$ let $h_i=(n_i -x_i, 1^{x_i})\in\mathcal{H}(n_i)$, for some $x_i\in [0,n_i-1]$. 
Finally, let $y=x-\sum_{i=1}^tx_i$. 
Then $$\mathcal{LR}(h; h_1,\ldots, h_t)= {t-1\choose y}.$$
\end{lemma}
\begin{proof}
We proceed by induction on $t$. If $t=2$, by Littlewood-Richardson rule we have that $\mathcal{LR}(h;h_1,h_2)=0$, unless $x_1+x_2\in \{x-1,x\}\cap \mathbb{N}$. 
In such cases we have that $\mathcal{LR}(h;h_1,h_2)=1$. These facts agree with the desired statement as $t-1=1$. 
Let us now suppose that $t\geq 3$ and let $K_1=(n_2+\dots +n_t-(x-x_1), 1^{x-x_1}),\ K_2=(n_2+\dots +n_t -(x-x_1-1),1^{x-x_1-1}) \in \mathcal{H}(n_2+\dots +n_t)$. Using the Littlewood-Richardson rule we observe that $\mathcal{LR}(h;h_1,\theta)=0$, unless $\theta\in\{K_1,K_2\}$. Moreover, $\mathcal{LR}(h;h_1, K_1)=\mathcal{LR}(h;h_1, K_2)=1$. By inductive hypothesis we know that $ \mathcal{LR}(K_1;h_2,\dots ,h_t)={t-2 \choose y} $ and $\mathcal{LR}(K_2;h_2,\dots ,h_t)={t-2 \choose y-1}$. Then, we can conclude that \[\mathcal{LR}(h; h_1,\ldots, h_t)={t-2 \choose y}+{t-2 \choose y-1}={t-1\choose y}.\]
\end{proof}

We give an example of the computation of the Littlewood-Richardson coefficient $\mathcal{LR}(\lambda; \mu, \nu)$ in the case $\lambda \in \mathcal{H}(2^n)$ and $\mu, \nu \in \mathcal{H}(2^{n-1})$. This case will be repeatedly used in the following section.
\begin{example}
	Let $n>1$ and $\lambda=(2^n-x,1^x)\in \mathcal{H}(2^n)$, where $x\in [1,2^n]$.
	We would like to restrict $\chi^\lambda$ to $Y:=\fS_{2^{n-1}} \times \fS_{2^{n-1}}$. 
	Using the Littlewood-Richardson rule we find the following decomposition:
	\[\chi^\lambda \down_{Y} = \chi^{(2^{n-1}-x,1^x)} \times \chi^{(2^{n-1})} + \sum_{y=0}^{x-1} \chi^{(2^{n-1}-y, 1^y)} \times \left(\chi^{(2^{n-1}-(x-y),1^{x-y})} + \chi^{(2^{n-1}-(x-y-1),1^{x-y-1})}\right).\]
	Thus we can compute the Littlewood-Richardson coefficient $\mathcal{LR}(\lambda; \mu, \nu)$. 
	Let $\mu=(2^{n-1}-t,1^t)$ and $\nu=(2^{n-1}-z,1^z)$.
	Define $V:=\Set{(x,0)}\cup \Set{(y,x-y),(y,x-y-1) \mid y\in[0,x-1] }$.
	Then \[\mathcal{LR}(\lambda; \mu, \nu) = \begin{cases}
		1, &\mbox{ if }(t,z)\in V;
		\\0, &\mbox{ otherwise.}
	\end{cases}\] 
	
	%Let $$V:=\lbrace ((2^{n-1}-x,1^x),(2^{n-1})) \rbrace \cup \lbrace ((2^{n-1}-y, 1^y),(2^{n-1}-(x-y),1^{x-y})), ((2^{n-1}-y, 1^y),(2^{n-1}-(x-y-1),1^{x-y-1})) \mid y\in[0,x-1] \rbrace.$$ Then
	%\[\mathcal{LR}(\lambda; \mu, \nu) = \begin{cases}
	%	1, &\mbox{ if }(\mu,\nu)\in V;
	%	\\0, &\mbox{ otherwise.}
	%\end{cases}\] 

	Notice that if $x=0$, $\lambda=(2^n)$ and $\chi^\lambda \down_{Y} = \chi^{(2^{n-1})}  \times \chi^{(2^{n-1})} $.
\end{example}

Let $n\in\mathbb{N}$ and let $n=2^{k_1}+\cdots+2^{k_t}$ be its binary expansion. Given $\phi\in\mathrm{Lin}(P_n)$, there exist unique $(h_1,\ldots, h_t)\in \mathcal{H}(n_1)\times\mathcal{H}(n_2)\times \cdots\times \mathcal{H}(n_t)$ such that $\phi=L_{h_1}\times\cdots\times L_{h_t}$. Here for all $i\in [1,t]$, $L_{h_i}$ denotes the only linear constituent of $\chi^{h_i} \down_{P_{2^{k_i}}}$, as described by Theorems \ref{thm: GMcKay} and \ref{lem: unique_linear}. 
In this case, we denote $\phi$ by $\phi(h_1,\ldots, h_t)$.

\begin{theorem}\label{thm: 4.3}
Let $n\in\mathbb{N}$ and let $k_1>\cdots >k_t\geq 0$ be such that $n=2^{k_1}+\cdots+2^{k_t}$. Let $x\in [0,n-1]$ and let
$h=(n-x,1^x)\in\mathcal{H}(n)$. For each $i\in [1,t]$ let $x_i\in [0, 2^{k_i}-1]$ and set $h_i=(2^{k_i}-x_i,1^{x_i})$. Finally let $y=x-\sum_{i=1}^tx_i$. Then $$[\chi^h \down_{P_n}, \phi(h_1,\ldots, h_t)]={t-1\choose y}.$$
\end{theorem}
\begin{proof}
For each $i\in [1,t]$ we let $n_i=2^{k_i}$. Set $\alpha=(n_1,\ldots, n_t)\in\mathcal{P}(n)$ and let $Y=\fS_\alpha$ be such that $P_n\leq Y$. 
Let $\mathcal{H}=\mathcal{H}(n_1)\times\mathcal{H}(n_2)\times \cdots\times \mathcal{H}(n_t)$ and $\phi= \phi(h_1,\ldots, h_t)$. 
From Lemma \ref{lem: hook-hooks} we know that $$(\chi^h) \down_Y=\sum_{(\lambda_1,\ldots, \lambda_t)\in\mathcal{H}}\mathcal{LR}(h; \lambda_1,\ldots,\lambda_t)(\chi^{\lambda_1}\times\cdots\times\chi^{\lambda_t}).$$
By Theorem \ref{thm: GMcKay}, $[(\chi^{\lambda_1}\times\cdots\times\chi^{\lambda_t}) \down_{P_n}, \phi]=0$ unless $(\lambda_1,\ldots, \lambda_t)=(h_1,\ldots, h_t)$. In the latter case we have $[(\chi^{h_1}\times\cdots\times\chi^{h_t}) \down_{P_n}, \phi]=1$. 
These observations used together with Lemma \ref{lem: diff} imply that $$[(\chi^h) \down _{P_n}, \phi]=[((\chi^h) \down_Y) \down_{P_n}, \phi]=\mathcal{LR}(h; h_1,\ldots, h_t)={t-1\choose y}.$$
\end{proof}

This theorem says in particular that for every natural number $n$ and every hook partition $h \in \mathcal{H}(n)$, the restriction $\chi^h \down_{P_n}$ has a linear constituent.
This was also \cite[Theorem 3.1]{GN}.

We conclude this section with a lemma that will prove to be useful in the second part of the paper. 
\begin{lemma}\label{lem: hook-1}
	Let $n\in \mathbb{N}$, $\lambda =\left(2^n-1,1\right)$. Then there exist characters $\theta_k \in \Irr(P_{2^n})$ for $k \in [0,n-1]$ such that $\theta_k(1)=2^k$ for each $k$ and $\chi^\lambda \down_{P_n}= \sum_{k=0}^{n-1} \theta_k$.
	 % for every $k \in [0,n-1]$ there exists a unique $\theta_k \in \Irr(P_{2^n})$ such that $\theta_k(1)=2^k$ and $\left[\chi^\lambda \down_{P_{2^n}}, \theta_k \right] \neq 0$. Moreover, $\left[\chi^\lambda \down_{P_{2^n}}, \theta_k \right]= 1$.
\end{lemma}
\begin{proof}
	We proceed by induction on $n$. If $n=1$ then $\lambda=(1^2)$ and $\chi^\lambda \down_{P_2} = \phi_1$, as desired. 
	If $n\geq 2$, by the Littlewood-Richardson rule we have 
	\[\chi^{\lambda} \down_{\fS_{2^{n-1}} \times \fS_{2^{n-1}}} = \chi^{(2^{n-1})} \times \chi^{(2^{n-1})} + \chi^{(2^{n-1})} \times \chi^{(2^{n-1}-1,1)} + \chi^{(2^{n-1}-1,1)} \times \chi^{(2^{n-1})}. \]
	However, $\chi^{(2^{n-1})} \down_{P_{2^{n-1}}} = 1_{P_{2^{n-1}}}$ and by inductive hypothesis, $\chi^{(2^{n-1}-1,1)} \down_{P_{2^{n-1}}} = \sum_{i=0}^{n-2} \psi_i$, where $\psi_i \in \Irr(P_{2^{n-1}})$ and $\psi_i(1)=2^i$. Notice also that $\psi_0\neq L_{(2^{n-1})}=1_{P_{2^{n-1}}}$. 
	Hence 
	\[\chi^\lambda \down_{P_{2^{n-1}} \times P_{2^{n-1}}} = \left( 1_{P_{2^{n-1}}} \times 1_{P_{2^{n-1}}} \right) + \sum_{i=0}^{n-2} \left( 1_{P_{2^{n-1}}} \times \psi_i + \psi_i \times 1_{P_{2^{n-1}}} \right) .\] 
	Therefore $\chi^\lambda \down_{P_{2^n}} = L_\lambda + \sum_{i=0}^{n-2} \left( 1_{P_{2^{n-1}}} \times \psi_i \right) \up^{P_{2^n}} $. The proof is then concluded by observing that $L_\lambda(1)=2^0$ and that $\left( 1_{P_{2^{n-1}}} \times \psi_i \right) \up^{P_{2^n}} (1)= 2^{i+1}$ for every $i\in [0, n-2]$.
\end{proof}

\section{On non-linear Sylow Branching Coefficients}\label{S: 5}

In Sections \ref{sec: 3} and \ref{S: 4} we completely described the linear constituents of the Sylow restriction of irreducible characters labelled by hook partitions. 
The aim of this section is to continue our investigation by focusing on irreducible constituents of higher degree. 
More precisely for any $k\in\mathbb{N}$ such that $2^k\in \cd(P_{n})$, we will devote our attention to study the structure of the set $\mathcal{H}_n^k$ defined as follows:  $$\mathcal{H}_n^k=\Set{\lambda \in \mathcal{H}(n) | \chi^\lambda \down_{P_{n}} \mbox{has an irreducible constituent of degree } 2^k }.$$
A first surprising result is stated in Theorem \ref{thm: inclusion}, where we show that $\mathcal{H}_n^k\subseteq \mathcal{H}_n^\ell$, whenever $\ell\leq k$. 
Then, in Theorem \ref{thm: generic_case} we prove that these sets have a very regular structure. More precisely, we show that for any $k\in\mathbb{N}$ as above, there exists a value $t\leq n$ such that $\mathcal{H}_n^k=\mathcal{B}_{n}(t) \cap \mathcal{H}(n)$.

In order to ease the notation, from now on given $n, t\in\mathbb{N}$ such that $t\leq n$
we will let $\bar{\mathcal{B}}_{n}(t)$ be the subset of $\mathcal{H}(n)$ defined by $$\bar{\mathcal{B}}_{n}(t) := \mathcal{B}_{n}(t) \cap \mathcal{H}(n).$$
It is important (and easy) to observe that the sets $\mathcal{H}_n^k$ and $\bar{\mathcal{B}}_{n}(t)$ are closed under conjugation of partitions.

%\begin{remark} \label{rmk: conj_closure}
%	For every $n,t,k$ as in the previous definition, $\Omega_{2^n}^k$ and $\mathcal{B}_{2^n}(t)$ are closed by conjugation, i.e. if $\lambda \in \Omega_{2^n}^k$ and %$\mu \in \mathcal{B}_{2^n}(t)$, then $\lambda' \in \Omega_{2^n}^k$ and $\mu' \in \mathcal{B}_{2^n}(t)$. Also $\mathcal{H}(2^n)$ is closed by conjugation.
%	Hence, the same holds for $\mathcal{H}_{2^n}^k$ and $\bar{\mathcal{B}}_{2^n}(t)$.
%
%	For every $n,t,k$ as in the previous definition, $\Omega_{n}^k$ and $\mathcal{B}_{n}(t)$ are closed by conjugation, i.e. if $\lambda \in \Omega_{n}^k$ and $\mu 
%\in \mathcal{B}_{n}(t)$, then $\lambda' \in \Omega_{n}^k$ and $\mu' \in \mathcal{B}_{n}(t)$. Also $\mathcal{H}(n)$ is closed by conjugation.
%	Hence, the same holds for $\mathcal{H}_{n}^k$ and $\bar{\mathcal{B}}_{n}(t)$.
%
%\end{remark}

We start the section with a short example. On one hand this will help the reader understand the behavior of the sets $\mathcal{H}_n^k$ for small values of $n$. On the other hand this will serve as base case of some of the later induction arguments. 

\begin{example}\label{ex: n=1,2,3}
Here we compute $\mathcal{H}_n^k$ whenever $n$ is a small power of $2$. More precisely we will just restrict our attention to the cases $n\in\{2,4,8\}$. 

If $n=2$ then there is not much to say as $\fS_2=P_2$. We just recall the notation introduced in Section \ref{sec: wr} and write $\Irr(P_2)=\Set{\phi_0, \phi_1}$ where $\phi_0$ is the trivial character of $P_2$. 

Let now $n=4$. Then $P_4\cong C_2 \wr C_2$ admits four linear characters $\mathcal{X}(i,j)$, for $i,j\in\{0,1\}$ and a unique irreducible character $\varphi$ of degree $2$. In particular, $\varphi=\left(\phi_0\times \phi_1\right)\up^{P_4}$. It follows that $\mathrm{cd}(P_4)=\{1,2\}$, hence we will be interested in computing the sets $\mathcal{H}_4^0$ and $\mathcal{H}_4^1$. In order to do this we are going to study the restriction to $P_4$ of those irreducible characters of $\fS_4$ that are labelled by partitions contained in the set $\mathcal{H}(4)=\{(4),(3,1),(2,1^2),(1^4)\}$. It is not difficult to see that: 
$$\chi^{(4)} \down_{P_4} = 1_{P_4}= \mathcal{X}(0,0)\ \text{and}\ \ \chi^{(3,1)}\down_{P_4}=\mathcal{X}(0,1) + \varphi.$$
Since the sets $\mathcal{H}_4^0$ and $\mathcal{H}_4^1$ are closed under conjugation of partitions, we conclude that 
$$\mathcal{H}_4^0=\mathcal{H}(4)=\bar{\mathcal{B}}_4(4),\ \text{and}\ \mathcal{H}_4^1=\{(3,1),(2,1^2)\}=\bar{\mathcal{B}}_4(3)$$

Finally, let us consider the case where $n=8$. Here $P_8=(P_4\times P_4)\rtimes P_2\cong P_4\wr P_2$. Since the base group $P_4\times P_4$ is naturally a subgroup of an appropriately chosen Young subgroup $Y\cong \fS_4\times \fS_4$ of $\fS_8$, our strategy is to restrict irreducible characters of $\fS_8$ to $Y$, to inductively deduce information on their restriction to $P_4\times P_4$ and finally to obtain results on their restriction to $P_8$. Consider for instance $\lambda=(6,1^2)$ and let $\chi=\chi^\lambda$. Using the Littlewood-Richardson rule we know that $\chi^{(4)}\times \chi^{(2,1^2)}$ is an irreducible constituent of $\chi\down_{Y}$. Using this, together with the calculations we have done for the case $n=4$ (inductive step) we deduce that $\mathcal{X}(0,0)\times \varphi$ is an irreducible constituent of  $\chi\down_{P_4\times P_4}$. We conclude that $(\mathcal{X}(0,0)\times \varphi)\up^{P_8}$ is an irreducible constituent (of degree $4$) of $\chi\down_{P_8}$. This shows that $(6,1^2)\in \mathcal{H}_8^2$. With completely similar arguments we obtain that 
$$\mathcal{H}_8^0=\bB_8(8)\ \ \text{and that}\ \mathcal{H}_8^1=\mathcal{H}_8^2=\bB_8(7).$$
This is all we needed to compute since $\mathrm{cd}(P_8)=\{1,2,4\}$. 
We conclude by mentioning that $\chi^{(6,1^2)} \down_{P_8}$ has a unique linear constituent, two distinct irreducible constituents of degree $2$ and three distinct irreducible constituents of degree $4$. On the other hand, $\chi^{(5,1^3)} \down_{P_8}$ has a unique linear constituent, three distinct irreducible constituents of degree $2$ and five distinct irreducible constituents of degree $4$. This can be easily verified using the strategy outlined above in the case of $\lambda=(6,1^2)$. We leave the explicit calculations to the interested reader. 
\end{example}

From now on we will denote by $\alpha_n$ the maximal integer $k$ such that $2^k$ is the degree of an irreducible character of $P_n$. This is formally defined and explained in the following Definition \ref{def: alphan} and Proposition \ref{prop: cd_n}. 

\begin{definition}\label{def: alphan}
For any natural number $t$ we define the integer $\alpha_{2^t}$ as follows. We set $$\alpha_1=\alpha_2=0,\ \ \alpha_4=1\ \text{and}\ \alpha_{2^t}=2^{t-2}+2^{t-3}-1\mbox{ for every }t\geq 3.$$ Notice in particular that $\alpha_{2^t}=2\alpha_{2^{t-1}}+1$, for any $t\geq 4$. 
Let now $n\in \mathbb{N}$, and $n=\sum_{i=1}^r 2^{n_i}$ be its binary expansion.
We set $\alpha_n\in\mathbb{N}$ to be defined as $\alpha_n := \sum_{i=1}^r \alpha_{2^{n_i}}$. 
\end{definition}

\begin{proposition} \label{prop: cd_n}
%	Let $t\in \mathbb{N}$.
%	\[ \cd(P_{2^t})= \Set{2^j|j\in [0,\alpha_{2^t}]}.\]
%	\\Moreover, if $t\geq 3$, in $\Irr(P_{2^t})$ there exist three distinct characters of degree $2^{\alpha_{2^t}}$.
	
	Given $n\in \mathbb{N}$, we have that $\cd(P_{n})= \Set{2^j|j\in [0,\alpha_n]}$. 
	Moreover, if $n\geq 8$, then $|\{\theta\in\mathrm{Irr}(P_n)\ |\ \theta(1)=2^{\alpha_{n}}\}|\geq 3$.
\end{proposition}
\begin{proof}
	Suppose first of all that $n=2^t$ is a power of $2$. We proceed by induction on $t$. If $t\in \lbrace 1,2,3 \rbrace$, the proposition holds as we can see in Example \ref{ex: n=1,2,3}.
	Let us now fix $t\geq 4$ and recall that $\alpha_{2^t}=2\alpha_{2^{t-1}}+1$. Let $k\in \left[0, \alpha_{2^t}-1 \right]$ and let $q \in [0, \alpha_{2^{t-1}} ]$, $r\in \lbrace 0,1 \rbrace$ be such that $k=2q+r$.
Suppose first that $r=0$, by inductive hypothesis there exists $\phi\in \Irr(P_{2^{t-1}})$ such that $\phi(1)=2^q$. Hence for every $\psi \in \Irr(P_2)$, $\mathcal{X}(\phi; \psi) \in \Irr(P_{2^t})$ has degree $2^k$.
	If instead $r=1$, then $q<\alpha_{2^{t-1}}$. If $q=0$ then $k=1$, and $|\Lin(P_{2^t})|=2^t>2$. Hence we can choose $L, L' \in \Lin(P_{2^t})$, $L\neq L'$ and $(L \times L') \up^{P_{2^t}}$ is an irreducible character of degree $2=2^k$.
	If $1\leq q < \alpha_{2^{t-1}}$, by inductive hypothesis there exist $\phi, \psi \in \Irr(P_{2^{t-1}})$ such that $\phi(1)=2^{q-1}$, $\psi(1)=2^{q+1}$. Hence $(\phi \times \psi) \up^{P_{2^t}} \in \Irr(P_{2^t})$ and it has degree $2^k$.
	If $k=\alpha_{2^t}$, by inductive hypothesis there exist $\psi_1$, $\psi_2$ and $\psi_3$ distinct irreducible characters of $P_{2^{t-1}}$ of degree $2^{\alpha_{2^{t-1}}}$. Since $\alpha_{2^t}=2\alpha_{2^{t-1}} +1$ we obtain that $\left(\psi_1 \times \psi_2 \right)\up^{P_{2^t}}, \left(\psi_2 \times \psi_3 \right)\up^{P_{2^t}}, \left(\psi_1 \times \psi_3 \right)\up^{P_{2^t}} \in \Irr(P_{2^t})$ are three distinct irreducible characters of degree $2^{\alpha_{2^t}}$.
	
Suppose now that $n\in\mathbb{N}$ is arbitrary and let $n=\sum_{i=1}^r 2^{n_i}$ with $n_1 \geq \cdots \geq n_r \geq 0$, be its binary expansion. From Section 
\ref{sec: wr} we know that $P_n = P_{2^{n_1}} \times \cdots \times P_{2^{n_r}}$ and therefore that $\Irr(P_n)=\Set{\phi_1 \times \cdots \times \phi_r | \phi_i \in \Irr(P_{2^{n_i}}),\ i=1,\dots ,r}$. Using this together with the information obtained above in the $2$-power case, we easily obtain that $\cd(P_{n})= \Set{2^j|j\in [0,\alpha_n]}$ and that  $|\{\theta\in\mathrm{Irr}(P_n)\ |\ \theta(1)=2^{\alpha_{n}}\}|\geq 3$, for any $n\geq 8$.
%
%as we can see in section \ref{sec: Sylow_sub}. Hence 
	%\[ \begin{split} \cd(P_n) &=\Set{2^h | h=d_1+\cdots +d_r,\ d_i \in \cd(P_{2^{n_i}}) \mbox{ for } i\in [1,r]} 
	%	\\&= \Set{2^h|  h=d_1+\cdots +d_r,\ d_i \in \left[0, \alpha_{2^{n_i}}\right]  \mbox{ for } i\in [1,r] } 
	%	\\&=\Set{2^h | h\in \left[0, \alpha_n\right]}, \end{split} \]
	%where the second equality holds by the $2$-power case and the third by definition of $\alpha_n$.
	%Finally, if $n\geq 3$ then $n_1\geq 3 $ and from the $2$-power case we know that there exist three distinct characters $\psi_1, \psi_2 , \psi_3$ in $\Irr(P_{2^{n_1}})$ of degree $2^{\alpha_{2^{n_1}}}$. Let $\theta$ be the product of some irreducible characters in $\Irr(P_{2^{n_i}})$ of degree $2^{\alpha_{2^{n_i}}}$ for every $i\in [2,r]$.	Then $\psi_j \times \theta$, $j\in [1,3]$, are three distinct irreducible characters of degree $2^{\alpha_n}$.
\end{proof}
%\begin{observation}
%	This result tells us, in particular, that the number of possible degrees for an irreducible character of the Sylow $2$-subgroup $P_{n}$ is $|\cd(P_{n})|=\alpha_n +1$.
%\end{observation}

We will state now a lemma that we will need for the proof of the main theorems of this section. To ease the notation we will denote by $Z^\lambda_\phi$ (instead of  $Z^{\chi^\lambda}_\phi$) the Sylow branching coefficient corresponding to the characters $\chi^\lambda\in\mathrm{Irr}(\fS_n)$ and $\phi\in\mathrm{Irr}(P_n)$.

Here, we start by giving some precise information concerning Sylow branching coefficients $Z^\lambda_\phi$, where $\phi$ is an irreducible character of the Sylow subgroup of degree $2$. In order to do this, it is convenient to introduce the following notation. 
For $\lambda=\left(2^n-x,1^x\right) \in \mathcal{H}(2^n)$ with $x\in [0, 2^n-1]$, we define $a_n^x\in\mathbb{N}$ as follows:
\[	a_n^x := | \Set{\phi \in \Irr(P_{2^n}) | \phi(1)=2 \mbox{ and } Z^\lambda_\phi\neq 0} | . \]

\begin{lemma}\label{lem: a_n^x}
	Let $n\geq 2$ and $\lambda=\left(2^n-x,1^x\right) \in \mathcal{H}(2^n)$.
	Then $a_n^x=\min \lbrace x, 2^n-1-x \rbrace$.
	%Then \[ a_n^x=\begin{cases}
	%	x, &\mbox{ if } 0\leq x\leq 2^{n-1}-1; \\ 2^n-x-1, &\mbox{ if } 2^{n-1}\leq x \leq 2^n-1.
	%\end{cases} \]
\end{lemma}

\begin{proof}
Recall that $P_{2^n}=B\rtimes P_2$, for some $B\leq \fS_{2^n}$ such that $B\cong P_{2^{n-1}}\times P_{2^{n-1}}$. Let $Y=\fS_{2^{n-1}}\times \fS_{2^{n-1}}\leq \fS_{2^n}$ be chosen such that $B\leq Y$. Let $\phi\in\mathrm{Irr}(P_{2^n})$ be such that $\phi(1)=2$. Then $\phi=(L_1\times L_2)\up_{B}^{P_{2^n}}$, for some $L_1,L_2\in\mathrm{Lin}(P_{2^{n-1}})$ such that $L_1\neq L_2$. In particular, we have that $Z^\lambda_{\phi}\neq 0$ if and only if $\left[\chi^\lambda\down_{B}, L_1\times L_2 \right]\neq 0$. Using this observation together with the Littlewood-Richardson rule and Theorem \ref{thm: GMcKay}, we deduce that $Z^\lambda_{\phi}\neq 0$ if and only if $\left[\chi^\lambda\down_{Y}, \chi_1\times \chi_2 \right]\neq 0$ where $\chi_i$ is the unique character in $\mathrm{Irr}_{\mathcal{H}}(\fS_{2^{n-1}})$ such that $\left[\chi_i\down_{P_{2^{n-1}}}, L_i\right]\neq 0$.

Let us first suppose that $x\leq 2^n-1-x$, in other words $x\in \left[0,2^{n-1}-1\right]$. If $x=0$, then $\lambda=(2^n)$ and $a_n^0=0$. 
Otherwise using the Littlewood-Richardson rule we have
	\begin{equation}\label{eq: dec_lambda}
		\begin{split}
		\chi^\lambda \down_{Y} = &\sum_{y=0}^m \chi^{\left(2^{n-1}-y, 1^y \right)} \times \left[  \chi^{\left( 2^{n-1}-(x-y), 1^{x-y} \right)} + \chi^{\left( 2^{n-1}-(x-y-1), 1^{x-y-1} \right)} \right] \\ + &\sum_{y=0}^m \left[  \chi^{\left( 2^{n-1}-(x-y), 1^{x-y} \right)} + \chi^{\left( 2^{n-1}-(x-y-1), 1^{x-y-1} \right)} \right]  \times \chi^{\left(2^{n-1}-y, 1^y \right)} , \end{split}
	\end{equation}
where $m=\lfloor x/2\rfloor$. From Equation (\ref{eq: dec_lambda}) we can see that there are exactly $x$ distinct unordered pairs $(\mu_1, \mu_2)\in \mathcal{H}(2^{n-1}) \times \mathcal{H}(2^{n-1}) $ such that $\mu_1\neq \mu_2$ and such that $\left[\chi^\lambda\down_{Y}, \chi^{\mu_1}\times \chi^{\mu_2} \right]\neq 0$. Using the observations discussed at the start of the proof we deduce that $a_n^x=x$. 

If $x\geq 2^n-1-x$ then we are in the case $x\in \left[2^{n-1}, 2^n -1\right]$ and we consider $\lambda'$, the conjugate partition of $\lambda$.
We observe that $\lambda' =\left(x+1, 1^{2^n-x-1}\right)$ and $0\leq 2^n-x-1 \leq 2^{n-1}-1$. Since $\chi^{\lambda'}\down_{P_{2^n}}= \chi^\lambda \down_{P_{2^n}} \cdot \chi^{\left(1^{2^n}\right)} \down_{P_{2^n}}$, we deduce that $a_n^x=a_n^{2^n-x-1}=2^n-x-1$, as desired. 
\end{proof}

Using this lemma we can notice that $a_n^x \neq 0$ if and only if $x \notin \{0,2^n-1 \}$, that is $\lambda \notin \{(2^n),(1^{2^n})\}$.
We deduce that $\mathcal{H}_{2^n}^1 = \bB_{2^n}(2^n-1)$ for every natural number $n\geq 2$.

We are now ready to state and prove one of the main results of this section. 

\begin{theorem}\label{thm: inclusion}
%	Let $n\in \mathbb{N}$. For every $l,k \in [0, \alpha_{2^n}]$ such that $l\leq k$, we have
%	\[\mathcal{H}_{2^n}^k \subseteq \mathcal{H}_{2^n}^l.\]
	
	Let $n \in \mathbb{N}$. For every $l,k \in [0, \alpha_{n}]$ such that $\ell\leq k$, we have
	\[\mathcal{H}_{n}^k \subseteq \mathcal{H}_{n}^{\ell}.\]
\end{theorem}

\begin{proof}
	Suppose first that $n=2^t$, for some $t\in\mathbb{N}$. Clearly, it is enough to show that 
	$$\mathcal{H}_{2^t}^k \subseteq \mathcal{H}_{2^t}^{k-1}, \mbox{ for every } k\in [1, \alpha_{2^t}].$$
	We proceed by induction on $t \geq 2$. If $t=2$, then the statement holds by direct computation (see Example \ref{ex: n=1,2,3}). 
Let $t\geq 3$ and let $\lambda \in \mathcal{H}_{2^t}^k$. By definition there exists an irreducible constituent $\theta$ of $\chi^\lambda \down_{P_{2^t}}$ of degree $2^k$.
	If $\theta=\mathcal{X}(\psi; \alpha)$ with $\psi \in \Irr(P_{2^{t-1}})$, $\psi(1)=2^\frac{k}{2}$ and $\alpha \in \Lin(P_2)$, then there exist $\mu, \nu \in \mathcal{H}(2^{t-1})$ such that $\mathcal{LR}(\lambda; \mu, \nu)\neq 0$ and $\psi$ is a constituent of both $\chi^\mu \down_{P_{2^{t-1}}}$ and $\chi^\nu \down_{P_{2^{t-1}}}$. In particular $\mu, \nu \in \mathcal{H}_{2^{t-1}}^{\frac{k}{2}}$.
	If $\frac{k}{2} >1$, using the inductive hypothesis for $\frac{k}{2}$ and then for $\frac{k}{2}-1$, we have that $\mu \in \mathcal{H}_{2^{t-1}}^{\frac{k}{2} -2}$. Therefore there exists an irreducible constituent $\psi'$ of $\chi^\mu \down_{P_{2^{t-1}}}$ of degree $2^{\frac{k}{2} -2}$. Hence we have that $(\psi' \times \psi) \up^{P_{2^t}}$ is an irreducible constituent of $\chi^\lambda \down_{P_{2^t}}$ of degree $2^{k-1}$ and we conclude that $\lambda \in \mathcal{H}_{2^t}^{k-1}$.
	Otherwise, if $\frac{k}{2}=1$ then $k=2$ and $\lambda \notin \{(2^t), (1^{2^t})\}$. Hence $\lambda\in \bB_{2^t}(2^{t}-1)$. By Lemma \ref{lem: a_n^x} we know that $\bB_{2^t}(2^{t}-1)=\mathcal{H}_{2^{t}}^{1}$. It follows that $\lambda \in \mathcal{H}_{2^{t}}^{k-1}= \mathcal{H}_{2^{t}}^{1}$.

	Suppose now that $\theta= \left(\theta_1 \times \theta_2\right) \up^{P_{2^t}}$ with $\theta_1, \theta_2 \in \Irr(P_{2^{t-1}})$, $\theta_1 \neq \theta_2$ and $\theta_1(1)=2^{h_1}$, $\theta_2(1)=2^{h_2}$ where $h_1+h_2=k-1$. Then there exist $\mu_1, \mu_2 \in \mathcal{H}(2^{t-1})$ such that $\mathcal{LR}(\lambda; \mu_1, \mu_2)\neq 0$ and $\theta_1$ (respectively $\theta_2$) is an irreducible constituent of $\chi^{\mu_1} \down_{P_{2^{t-1}}}$ (respectively $\chi^{\mu_2} \down_{P_{2^{t-1}}}$). In particular, $\mu_i \in \mathcal{H}_{2^{t-1}}^{h_i}$ for $i=1,2$.
	Suppose $h_1=h_2=\frac{k-1}{2}:=h$. If $h=0$ then $k=1$ and we know that $\lambda \in \mathcal{H}_{2^t}^0=\mathcal{H}_{2^{t}}^{k-1}$ by Theorem \ref{thm: 4.3}. If $h>0$ then by induction we have that $\mu_1 \in \mathcal{H}_{2^{t-1}}^{h} \subseteq \mathcal{H}_{2^{t-1}}^{h-1}$. Hence there exists an irreducible constituent $\psi$ of $\chi^{\mu_1} \down_{P_{2^{t-1}}}$ of degree $2^{h-1}$. Therefore $\left(\psi \times \theta_2\right) \up^{P_{2^t}}$ is an irreducible constituent of $\chi^{\lambda} \down_{P_{2^t}}$ of degree $2^{k-1}$, and so $\lambda \in \mathcal{H}_{2^{t}}^{k-1}$.
	
	If instead $h_1 \neq h_2$, we can suppose without loss of generality that $h_1 >h_2$. In particular $h_1 \geq 1 $, so we can use the inductive hypothesis and we have $\mu_1 \in \mathcal{H}_{2^{t-1}}^{h_1 -1}$. Hence there exists an irreducible constituent $\psi$ of $\chi^{\mu_1} \down_{P_{2^{t-1}}}$ of degree $2^{h_1 -1}$.
	If $\psi \neq \theta_2$ then considering $(\psi \times \theta_2)\up^{P_{2^t}}$ we deduce that $\lambda \in \mathcal{H}_{2^{t}}^{k-1}$, since $\theta_2(1)=2^{h_2}$ and $h_1+h_2=k-1$.
	If instead $\psi =\theta_2$ then in particular $h_2=h_1-1$. Suppose $h_2>0$ and use the inductive hypothesis. We have $\mu_2 \in \mathcal{H}_{2^{t-1}}^{h_2}\subseteq \mathcal{H}_{2^{t-1}}^{h_2-1}$. Hence there exists $\psi' \in \Irr(P_{2^{t-1}})$ such that $\left[\psi', \chi^{\mu_2} \down_{P_{2^{t-1}}}\right]\neq 0$ and $\psi'(1)=2^{h_2-1}$. We conclude that $\lambda \in \mathcal{H}_{2^t}^{k-1}$ since $(\theta_1 \times \psi' )\up^{P_{2^t}}$ is an irreducible constituent of $\chi^\lambda \down_{P_{2^{t}}}$ of degree $2^{k-1}$.
	If $h_2=0$ then $k=2$ and arguing as before, using Lemma \ref{lem: a_n^x}, we deduce that $\lambda \in \mathcal{H}_{2^t}^1 =\mathcal{H}_{2^t}^{k-1}$.
	
	Let now $n\in \mathbb{N}$ and let $n=\sum_{i=1}^r 2^{n_i}$ be its binary expansion. %Hence $P_n= P_{2^{n_1}}\times \cdots \times P_{2_{n_r}}$ and $\Irr(P_n)=\Set{\phi_1 \times \cdots \times \phi_r | \phi_i \in \Irr(P_{2^{n_i}}) }$.
	Let $\lambda \in \mathcal{H}_n^k$, then there exists $\phi \in \Irr(P_n)$ such that $\left[\phi , \chi^\lambda \down_{P_{n}} \right] \neq 0$ and $\phi(1)=2^k$. By the structure of the irreducible characters of $P_n$ in section \ref{sec: wr}, $\phi=\phi_1 \times \cdots \times \phi_r$, with $\phi_i \in \Irr(P_{2^{n_i}})$ and $\phi_i(1)=2^{j_i}$ for every $i=1,\dots ,r$, where $j_1+ \cdots + j_r=k$. Therefore for every $i=1,\dots ,r$ there exists $\mu_i \in \mathcal{H}(2^{n_i})$ such that $\phi_i$ is an irreducible constituent of $\chi^{\mu_i} \down_{P_{2^{n_i}}}$ and $\mathcal{LR}(\lambda; \mu_1, \dots , \mu_r)\neq 0$.
	Let $\ell\leq k$ and for every $i=1, \dots , r$ let $d_i \in \left[0, \alpha_{2^{n_i}}\right]$, $d_i \leq j_i$ such that $d_1+\cdots +d_r =\ell$. By previous case, for every $i=1, \dots , r$, $\mu_i \in \mathcal{H}_{2^{n_i}}^{j_i} \subseteq \mathcal{H}_{2^{n_i}}^{d_i}$. Hence there exists an irreducible constituent $\psi_i$ of $\chi^{\mu_i} \down_{P_{2^{n_i}}}$ of degree $2^{d_i}$. Therefore $\psi_1 \times \cdots \times \psi_r$ is an irreducible constituent of $\chi^\lambda \down_{P_{n}}$ of degree $2^{\ell}$, and so $\lambda \in \mathcal{H}_n^{\ell}$.	
\end{proof}

We introduce here a combinatorial operation between hook partitions that is very similar to the $\star$ operation described before Lemma \ref{lem: B star}.

\begin{definition}
	Let $n, m \in \mathbb{N}$ and $A\subseteq \mathcal{P}(n)$, $B\subseteq \mathcal{P}(m)$.
	\[ A \Diamond  B := (A\star B) \cap \mathcal{H}(n+m). \]
\end{definition}

\begin{lemma}\label{lem: B-diamond}
	Let $n, n', t, t' \in \mathbb{N}$ be such that $\frac{n}{2} < t \leq n$ and $\frac{n'}{2} < t' \leq n'$. Then
	\[\bar{\mathcal{B}}_n(t) \Diamond \bar{\mathcal{B}}_{n'}(t') = \bar{\mathcal{B}}_{n+n'}(t+t'). \]
\end{lemma}
\begin{proof}
	By definition, we have $\bar{\mathcal{B}}_n(t)\subseteq \mathcal{B}_n(t)$ and $\bar{\mathcal{B}}_{n'}(t')\subseteq \mathcal{B}_{n'}(t')$. Hence 
	\begin{equation*}
		\begin{split} \bar{\mathcal{B}}_n(t) \Diamond \bar{\mathcal{B}}_{n'}(t') &\subseteq \mathcal{B}_n(t) \Diamond \mathcal{B}_{n'}(t') = \left( \mathcal{B}_n(t) \star \mathcal{B}_{n'}(t') \right) \cap \mathcal{H}(n+n') \\&= \mathcal{B}_{n+n'}(t+t') \cap \mathcal{H}(n+n') = \bar{\mathcal{B}}_{n+n'}(t+t'), \end{split} 
	\end{equation*}
	where the second equality holds by Lemma \ref{lem: B star}.
	To prove the other inclusion it is enough to show that $\mathcal{B}_n(t) \Diamond \mathcal{B}_{n'}(t') \subseteq \bar{\mathcal{B}}_n(t) \Diamond \bar{\mathcal{B}}_{n'}(t')$.
	Let $\lambda \in \mathcal{B}_n(t) \Diamond \mathcal{B}_{n'}(t')$, then $\lambda \in \mathcal{H}(n+n')$ and there exist $\mu \in \mathcal{B}_n(t),\ \mu' \in \mathcal{B}_{n'}(t')$ such that $\mathcal{LR}(\lambda ; \mu, \mu') \neq 0$. However by Lemma \ref{lem: hook-hooks}, $\mu$ and $\mu'$ have to be hook partitions. Therefore $\mu \in \bar{\mathcal{B}}_n(t)$, $\mu' \in \bar{\mathcal{B}}_{n'}(t')$ and $\lambda \in \bar{\mathcal{B}}_n(t) \Diamond \bar{\mathcal{B}}_{n'}(t')$. 
\end{proof}

\begin{lemma}\label{lem: O_diamond_2power}
	Let $n\geq 2$ and $i,j\in \left[0, \alpha_{2^{n-1}}\right]$ be such that $i\neq j$. Then
	\[\mathcal{H}_{2^{n-1}}^i \Diamond \mathcal{H}_{2^{n-1}}^j \subseteq \mathcal{H}_{2^n}^{i+j+1} .\]
\end{lemma}
\begin{proof}
	Let $\lambda \in \mathcal{H}_{2^{n-1}}^i \Diamond \mathcal{H}_{2^{n-1}}^j$. By definition $\lambda \in \mathcal{H}(2^n)$ and there exist $\mu \in \mathcal{H}_{2^{n-1}}^i,\ \nu \in \mathcal{H}_{2^{n-1}}^j $ such that $\mathcal{LR}(\lambda; \mu, \nu)\neq 0$.
	Hence there exist $\phi, \psi \in \Irr(P_{2^{n-1}})$ such that $\left[ \phi, \chi^\mu \down_{P_{2^{n-1}}} \right]\neq 0$, $\left[\psi, \chi^\nu \down_{P_{2^{n-1}}}\right]\neq 0$ and $\phi(1)=2^i$, $\psi(1)=2^j$.
	Since $i\neq j$, we have that $\theta=\left( \phi \times \psi \right) \up^{P_{2^n}}\in \Irr(P_{2^n})$. Moreover, $\theta(1)=2^{i+j+1}$ and $\left[\theta , \chi^\lambda \down_{P_{2^n}}\right]\neq 0$. It follows that $\lambda \in \mathcal{H}_{2^n}^{i+j+1}$.
\end{proof}

\begin{lemma}\label{lem: O_diamond_general}
	Let $n\in \mathbb{N}$ and let $n=\sum_{i=1}^r 2^{n_i}$ be its binary expansion. Suppose that for every $i=1, \dots , r$, $j_i \in \left[0, \alpha_{2^{n_i}}\right]$ is such that $j_1+ \cdots +j_r =k\in \left[0, \alpha_n \right]$. Then
	\[ \mathcal{H}_{2^{n_1}}^{j_1} \Diamond \cdots \Diamond \mathcal{H}_{2^{n_r}}^{j_r} \subseteq \mathcal{H}_n^k .\]
\end{lemma}
\begin{proof}
	Let $\lambda \in \mathcal{H}_{2^{n_1}}^{j_1} \Diamond \cdots \Diamond \mathcal{H}_{2^{n_r}}^{j_r} $. By definition, for every $i=1, \dots ,r$ there exists $\mu_i \in \mathcal{H}_{2^{n_i}}^{j_i} $ such that $\mathcal{LR}(\lambda; \mu_1, \dots , \mu_r) \neq 0$. Hence there exists an irreducible constituent $\phi_i$ of $\chi^{\mu_i} \down_{P_{2^{n_i}}} $ of degree $2^{j_i}$ for every $i=1, \dots ,r$. Since $P_n=P_{2^{n_1}}\times\cdots\times P_{2^{n_r}}\leq \fS_{2^{n_1}}\times\cdots\times \fS_{2^{n_r}}\leq\fS_n$, it follows that $\phi_1 \times \cdots \times \phi_r$ is an irreducible constituent of $\chi^\lambda \down_{P_{n}}$ of degree $2^k$. Hence $\lambda \in \mathcal{H}_n^k$.
\end{proof}

We are now ready to state the second main result of the section. In particular, we are able to show that the sets $\mathcal{H}_{n}^k$ have a very regular structure. In Theorem \ref{thm: 2_power} we first deal with the case where $n$ is a power of $2$. Then in Theorem \ref{thm: generic_case} we show that for any $n\in\mathbb{N}$ and any $k\in [0,\alpha_n]$ there exists $T_n^k\in\mathbb{N}$ such that $\mathcal{H}_{n}^k=\bB_{n}(T_n^k)$.

\begin{theorem} \label{thm: 2_power}
	Let $n \in \mathbb{N}$ and $k\in [0, \alpha_{2^n}]$. Then:
	\begin{itemize}
		\item[(1)] there exists $t_n^k \in \left[ 2^{n-1}+1, 2^n \right]$ such that $\mathcal{H}_{2^n}^k = \bB_{2^n}(t_n^k)$;
		\item[(2)] if $k>1$, then for every $\lambda \in \bB_{2^n}(t_n^k -1)$, $\chi^\lambda \down_{P_{2^n}}$ has at least three distinct irreducible constituents of degree $2^k$.
	\end{itemize}
\end{theorem}
\begin{proof}
	We proceed by induction on $n$. If $n\in \{1,2\}$ then the theorem holds (see Example \ref{ex: n=1,2,3}).
	For $n\geq 3$ we proceed by induction on $k$. By Theorem \ref{thm: GMcKay} we know that $\mathcal{H}_{2^n}^0=\bB_{2^n}(2^n)$ and hence that $t_n^0=2^n$.
	If instead $k=1$, by Lemma \ref{lem: a_n^x} we have
	that $\mathcal{H}_{2^n}^1 = \bB_{2^n}(2^n-1)$ and hence that $t_n^1=2^n-1$. 
	Suppose now that $k=2$. We want to show that $\mathcal{H}_{2^{n}}^2 = \bB_{2^n}(2^n -1)$. We have that 
	\[\bB_{2^n}(2^n-1)= \bB_{2^{n-1}}(2^{n-1}) \Diamond \bB_{2^{n-1}}(2^{n-1}-1) = \mathcal{H}_{2^{n-1}}^0 \Diamond \mathcal{H}_{2^{n-1}}^1 \subseteq \mathcal{H}_{2^n}^2 , \]
	where the first equality holds by Lemma \ref{lem: B-diamond} and where the last inclusion holds by Lemma \ref{lem: O_diamond_2power}.
	On the other hand, we clearly have that $\mathcal{H}_{2^n}^2 \subseteq \bB_{2^n}(2^n-1)$, because $(2^n), (1^{2^n})\notin \mathcal{H}_{2^n}^2$. 
%	Hence we only need to prove that $\mathcal{H}_{2^n}^2 \subseteq \bB_{2^n}(2^n-1)$. Let $\lambda \in \mathcal{H}_{2^n}^2$, then there exists an irreducible constituent $\phi$ of $\chi^\lambda \down_{P_{2^n}}$ of degree $2^2$.
%	Suppose $\phi = \mathcal{X}(\psi, \alpha)$, with $\psi \in \Irr(P_{2^{n-1}})$, $\psi(1)=2$ and $\alpha \in \Lin(P_2)$. Then there exist $\mu, \nu \in \mathcal{H}(2^{n-1})$ such that $\mathcal{LR}(\lambda; \mu, \nu)\neq 0$ and $\psi$ is an irreducible constituent both for $\chi^{\mu} \down_{P_{2^{n-1}}}$ and $\chi^{\nu} \down_{P_{2^{n-1}}}$.
%	By the previous case $k=1$ we know that $\mathcal{H}_{2^{n-1}}^1=\bB_{2^{n-1}}(2^{n-1}-1)$. Therefore $\mu,\nu \in \bB_{2^{n-1}}(2^{n-1}-1)$, and $\lambda \in \bB_{2^{n-1}}(2^{n-1}-1) \Diamond \bB_{2^{n-1}}(2^{n-1}-1) = \bB_{2^n}(2^n-2) \subseteq \bB_{2^n}(2^n-1) $, by Lemma \ref{lem: B-diamond} and Theorem \ref{thm: inclusion}.
%	
%	Let instead $\phi=(\phi_1 \times \phi_2) \up^{P_{2^n}}$ with $\phi_1, \phi_2 \in \Irr(P_{2^{n-1}})$ and without loss of generality $\phi_1(1)=1$, $\phi_2(1)=2$. Then there exist $\mu_1, \mu_2 \in \mathcal{H}(2^{n-1})$ such that $\mathcal{LR}(\lambda; \mu_1, \mu_2)\neq 0$ and $\phi_1$ (resp. $\phi_2$) is an irreducible constituent of $\chi^{\mu_1} \down_{P_{2^{n-1}}}$ (resp. $\chi^{\mu_2} \down_{P_{2^{n-1}}}$). Hence $\mu_1 \in \mathcal{H}_{2^{n-1}}^0=\bB_{2^{n-1}}(2^{n-1})$ and $\mu_2 \in \mathcal{H}_{2^{n-1}}^1=\bB_{2^{n-1}}(2^{n-1}-1)$. Therefore $\lambda \in \bB_{2^{n-1}}(2^{n-1}) \Diamond \bB_{2^{n-1}}(2^{n-1}-1) = \bB_{2^n}(2^n-1)$, by Lemma \ref{lem: B-diamond}.
To conclude, we need to show that for every $\lambda \in \bB_{2^{n}}(t_n^2-1)=\bB_{2^n}(2^n-2)$, $\chi^\lambda \down_{P_{2^n}}$ has three distinct irreducible constituents of degree $2^2$. If $\lambda_1 < 2^n-2$ then $\lambda \in \bB_{2^{n}}(2^n-3)= \bB_{2^{n-1}}(2^{n-1}-3) \Diamond \bB_{2^{n-1}}(2^{n-1})$, by Lemma \ref{lem: B-diamond}. Hence there exist $\mu \in \bB_{2^{n-1}}(2^{n-1}-3)$ and $\nu \in \bB_{2^{n-1}}(2^{n-1})$ such that $\mathcal{LR}(\lambda; \mu, \nu)\neq 0$. In particular, if $\mu_1= 2^{n-1}-x \leq 2^{n-1}-3$ then using Lemma \ref{lem: a_n^x}, we deduce that there exist three distinct irreducible constituents $\theta_1, \theta_2$ and $\theta_3$ of $\chi^\mu \down_{P_{2^{n-1}}}$ of degree $2$. On the other hand, $\nu \in \bB_{2^{n-1}}(2^{n-1})=\mathcal{H}_{2^{n-1}}^0$. Hence there exists a linear constituent $L$ of $\chi^\nu \down_{P_{2^{n-1}}}$. Therefore $\left( \theta_1 \times L \right) \up^{P_{2^n}}, \left( \theta_2 \times L \right) \up^{P_{2^n}}$ and $\left( \theta_3 \times L\right) \up^{P_{2^n}}$ are three distinct irreducible contituents of $\chi^\lambda \down_{P_{2^n}}$ of degree $2^2$.
	Consider now the case $\lambda_1 =2^n-2$, then $\mathcal{LR}(\lambda; (2^{n-1}-1,1), (2^{n-1}-1,1))\neq 0\neq \mathcal{LR}(\lambda; (2^{n-1}-2, 1^2), (2^{n-1}))$. By Lemma \ref{lem: hook-1}, we know that $\chi^{(2^{n-1}-1,1)} \down_{P_{2^{n-1}}}$ has an irreducible consituent $\theta $ of degree $2$. Hence there exists $\alpha\in\mathrm{Irr}(P_2)$ such that $\mathcal{X}(\theta; \alpha)$ is an irreducible constituent of $\chi^\lambda \down_{P_{2^n}}$ of degree $4$. 
	Moreover, by Lemma \ref{lem: a_n^x} we know that $\chi^{(2^{n-1}-2,1^2)} \down_{P_{2^{n-1}}}$ admits two distinct irreducible constituents $\psi_1$ and $\psi_2$ of degree $2$. We conclude that $\left(\psi_1 \times 1_{P_{2^{n-1}}} \right) \up^{P_{2^n}}$, $\left(\psi_2 \times 1_{P_{2^{n-1}}} \right) \up^{P_{2^n}}$ and $\mathcal{X}(\theta; \alpha)$ are three distinct irreducible constituents of $\chi^\lambda \down_{P_{2^n}}$ of degree $2^2$.

Let us now suppose that $k\geq 3$. From now on we will denote $w:=\frac{k-1}{2}$.
We define $M\in\mathbb{N}$ as follows. 
	\[M:=\max \Set{t_{n-1}^i +t_{n-1}^j , 2t_{n-1}^{w}+ \delta_{n-1}^{w} | i,j,w \in \left[0, \alpha_{2^{n-1}} \right],\ i+j=k-1 \mbox{ and } i\neq j }, \]
	where for $h\in \left[0, \alpha_{2^{n-1}} \right]$,
		\[ \begin{split} &t_{n-1}^h \mbox{ is defined inductively such that }\mathcal{H}_{2^{n-1}}^h = \bB_{2^{n-1}}(t_{n-1}^h); 
		\\ &\lambda_{n-1}^h := \left(t_{n-1}^h , 1, \dots , 1 \right) \in \bB_{2^{n-1}}(t_{n-1}^h); \mbox{ and}
		\\ &\delta_{n-1}^h := \begin{cases}
			0, &\mbox{ if } \chi^{\lambda_{n-1}^h} \down_{P_{2^{n-1}}} \mbox{ has two distinct irreducible constituents of degree } 2^h;
			\\-1, &\mbox{ if } \chi^{\lambda_{n-1}^h} \down_{P_{2^{n-1}}} \mbox{ has a unique irreducible constituent of degree } 2^h.
		\end{cases} \end{split}\]
	We will now show that $M=t_n^k$, or equivalently that $\mathcal{H}_{2^n}^k = \bB_{2^n}(M)$. 
	To show that $\mathcal{H}_{2^n}^k \supseteq \bB_{2^n}(M)$, we need to split our discussion into three cases, depending on the value $M$.
	\begin{enumerate}
		\item First, let us suppose that $M=t_{n-1}^i +t_{n-1}^j$, for some $i,j  \in \left[0, \alpha_{2^{n-1}} \right]$, $i+j=k-1$ and $i\neq j$.
	We have
	\[\bB_{2^n}(M)=\bB_{2^{n-1}}(t_{n-1}^i) \Diamond \bB_{2^{n-1}}(t_{n-1}^j) = \mathcal{H}_{2^{n-1}}^i \Diamond \mathcal{H}_{2^{n-1}}^j \subseteq \mathcal{H}_{2^n}^k , \]
	respectively by Lemma \ref{lem: B-diamond}, inductive hypothesis and Lemma \ref{lem: O_diamond_2power}.
	Moreover, since $k \geq 3$, without loss of generality we can assume that $i>1$. Hence $\bB_{2^n}(M-1) = \bB_{2^{n-1}}(t_{n-1}^i-1) \Diamond \bB_{2^{n-1}}(t_{n-1}^j)$ by Lemma \ref{lem: B-diamond}. If $\lambda \in \bB_{2^n}(M-1)$, then there exist $\mu \in \bB_{2^{n-1}}(t_{n-1}^i-1)$ and $\nu \in \bB_{2^{n-1}}(t_{n-1}^j)$ such that $\mathcal{LR}(\lambda; \mu, \nu)\neq 0$. By inductive hypothesis there exist three distinct irreducible constituents $\theta_1$, $\theta_2$ and $\theta_3$ of $\chi^\mu \down_{P_{2^{n-1}}}$ of degree $2^i$, and by definition there exists an irreducible constituent $\psi$ of $\chi^\nu \down_{P_{2^{n-1}}}$ of degree $2^j$. It follows that $(\theta_1 \times \psi) \up^{P_{2^n}}$, $(\theta_2 \times \psi) \up^{P_{2^n}}$ and $(\theta_3 \times \psi) \up^{P_{2^n}}$ are three distinct irreducible constituents of $\chi^\lambda \down_{P_{2^n}}$ of degree $2^k$.

\item For the second case we assume that $\delta_{n-1}^{w}=0$ and $M=2 t_{n-1}^{w}$. %k necessariamente >3, serve per usare l'ipotesi induttiva (2) nel caso centrale
In this setting we observe that $k$ must be strictly greater than $3$. In fact, if $k=3$ then Example \ref{ex: n=1,2,3} and Lemma \ref{lem: a_n^x} show that 
$M=2t_{n-1}^1=2(2^{n-1}-1)=2^{n}-2$. On the other hand, by definition of $M$ we know that $M=\max \Set{t_{n-1}^0 + t_{n-1}^2 , 2^n-2}=t_{n-1}^0 + t_{n-1}^2=2^{n-1}+ 2^{n-1}-1 = 2^n-1$ and this is a contradiction. Hence we can assume that $k>3$. Let $\lambda \in \bB_{2^n}(M)$.
	If $\lambda_1=M$, then $\mathcal{LR}\left(\lambda; \lambda_{n-1}^w , \lambda_{n-1}^{w} \right)\neq 0$. Recall that $\delta_{n-1}^{w}=0$ means that $\chi^{\lambda_{n-1}^w \down_{P_{2^{n-1}}}}$ has two distinct irreducible constituents $\theta_1$ and $\theta_2$ of degree $2^{w}$. Hence $(\theta_1 \times \theta_2) \up^{P_{2^n}}$ is an irreducible constituent of $\chi^{\lambda} \down_{P_{2^n}}$ of degree $2^k$ and therefore $\lambda \in \mathcal{H}_{2^n}^k$.
			
	If $t_{n-1}^{w} \leq \lambda_1 <M$, then $\mathcal{LR}\left(\lambda; \lambda_{n-1}^{w}, \mu \right)\neq 0$ for some $\mu \in \mathcal{H}(2^{n-1})$, with $\mu_1=\lambda_1 - t_{n-1}^{w} < M-t_{n-1}^{w} =t_{n-1}^{w}$. In particular, $\mu \in \bB_{2^{n-1}}\left(t_{n-1}^{w} -1\right)$. Since $w=\frac{k-1}{2} >1$, by inductive hypothesis there exist three distinct irreducible constituents $\psi_1, \psi_2$ and $\psi_3$ of $\chi^\mu \down_{P_{2^{n-1}}}$ of degree $2^{w}$. 			
	Let $\theta_1$ and $\theta_2$ be as in the previous case, and suppose without loss of generality that $\theta_1 \notin \Set{\psi_2, \psi_3}$ and that $\theta_2 \neq \psi_1$. %$\psi_1 =\theta_1$ e $\psi_2 = \theta_2$, il peggiore dei casi
	Then $(\theta_1 \times \psi_2) \up^{P_{2^n}}, (\theta_1 \times \psi_3)\up^{P_{2^n}}$ and $(\theta_2 \times \psi_1)\up^{P_{2^n}}$ are three distinct irreducible constituents of $\chi^\lambda \down_{P_{2^n}}$ of degree $2^k$. In particular, $\lambda \in \mathcal{H}_{2^n}^k$.
			
	Suppose now that $1 \leq \lambda_1 < t_{n-1}^{w}$, then we have 
	\[\begin{split} (\lambda')_1 &=2^n+1-\lambda_1 > 2^n+1-t_{n-1}^{w} \geq 2^n+1-2^{n-1} \\&=2^{n-1}+1 \geq t_{n-1}^{\frac{k+1}{2}} +1 > t_{n-1}^{\frac{k+1}{2}}. \end{split} \]
	Since $\lambda' \in \bB_{2^n}(M)$, from the previous case we deduce that $\lambda'\in\mathcal{H}_{2^n}^k$ and we conclude that $\lambda\in\mathcal{H}_{2^n}^k$, as $\mathcal{H}_{2^n}^k$ is closed under conjugation of partitions.

\item Finally consider the case where $\delta_{n-1}^{w}=-1$ and $M=2t_{n-1}^{w}-1$.	Arguing exactly as above we observe that $k> 3$ and hence that $w=\frac{k-1}{2}>1$. Moreover, in this case we have that $\chi^{\lambda_{n-1}^{w}} \down_{P_{2^{n-1}}}$ has a unique irreducible constituent $\psi$ of degree $2^{w}$. Let us fix $\lambda \in \bB_{2^n}(M)$.
	If $\lambda_1=M$, then $\mathcal{LR}\left(\lambda; \lambda_{n-1}^{w} , \mu \right)\neq 0$, for some $\mu \in \bB_{2^{n-1}}\left(t_{n-1}^{w}-1\right)$. Indeed $\bB_{2^n}(M) = \bB_{2^{n-1}}\left(t_{n-1}^{w}\right) \Diamond \bB_{2^{n-1}}\left(t_{n-1}^{w}-1\right)$, by Lemma \ref{lem: B-diamond}. Using the inductive hypothesis on $\mu$, we have that there exist three distinct irreducible constituents $\theta_1, \theta_2$ and $\theta_3$ of $\chi^\mu \down_{P_{2^{n-1}}}$ of degree $2^{w}$. Without loss of generality we can suppose that $\psi \neq \theta_1$, hence $(\psi \times \theta_1)\up^{P_{2^n}}$ is an irreducible constituent of $\chi^\lambda \down_{P_{2^n}}$ of degree $2^k$. Therefore $\lambda \in \mathcal{H}_{2^n}^k$.
				
	If instead $\lambda_1 < M $, by Lemma \ref{lem: B-diamond} we have
	\[\lambda \in \bB_{2^n}(M-1) = \bB_{2^{n-1}}\left(t_{n-1}^{w} -1 \right) \Diamond \bB_{2^{n-1}}\left(t_{n-1}^{w} -1 \right) .\]
	Hence there exist $\mu, \nu \in \bB_{2^{n-1}}\left(t_{n-1}^{w} -1 \right)$ such that $\mathcal{LR}(\lambda; \mu, \nu)\neq 0$. By induction there exist three distinct irreducible constituents $\theta_1, \theta_2$ and $\theta_3$ of $\chi^\mu \down_{P_{2^{n-1}}}$ and three distinct ones $\sigma_1, \sigma_2$ and $\sigma_3$ of $\chi^\nu \down_{P_{2^{n-1}}}$, all of them of degree $2^{w}$. Without loss of generality we can suppose that $\theta_1 \notin \Set{\sigma_2, \sigma_3}$ and $\theta_2 \neq \sigma_3$. Hence $(\theta_1 \times \sigma_2) \up^{P_{2^n}}, (\theta_1 \times \sigma_3) \up^{P_{2^n}}$ and $(\theta_2 \times \sigma_3) \up^{P_{2^n}}$ are three distinct irreducible constituents of $\chi^\lambda \down_{P_{2^n}}$ of degree $2^k$. In particular, $\lambda \in \mathcal{H}_{2^n}^k$.
			
\end{enumerate}
		
\vspace{0.3cm}

	We need now to prove that $\mathcal{H}_{2^n}^k \subseteq \bB_{2^n}(M)$. To do this, suppose by contradiction that there exists $\lambda \in \mathcal{H}_{2^n}^k \smallsetminus \bB_{2^n}(M)$ and without loss of generality suppose also that $\lambda_1 \geq M+1$.
	By definition there exists $\phi \in \Irr(P_{2^n})$ such that $\left[\chi^\lambda \down_{P_{2^n}}, \phi \right]\neq 0$ and $\phi(1)=2^k$.
	\begin{enumerate}
	\item Suppose first that $\phi= \mathcal{X}(\psi, \alpha)$ with $\psi \in \Irr(P_{2^{n-1}})$, $\psi(1)=2^{\frac{k}{2}}$ and $\alpha \in \Lin(P_2)$. Hence there exist $\mu, \nu \in \mathcal{H}(2^{n-1})$ such that $\mathcal{LR}(\lambda; \mu, \nu)\neq 0$ and $\psi$ is both an irreducible constituent of $\chi^\mu \down_{P_{2^{n-1}}}$ and of $\chi^\nu \down_{P_{2^{n-1}}}$. Therefore $\mu, \nu \in \mathcal{H}_{2^{n-1}}^{\frac{k}{2}} $.
	By inductive hypothesis and by Theorem \ref{thm: inclusion}, we have
	\[\bB_{2^{n-1}}\left(t_{n-1}^{\frac{k}{2}}\right)=\mathcal{H}_{2^{n-1}}^{\frac{k}{2}} \subseteq \mathcal{H}_{2^{n-1}}^{\frac{k}{2} -1} = \bB_{2^{n-1}}\left(t_{n-1}^{\frac{k}{2} -1}\right). \]
	In particular, $t_{n-1}^{\frac{k}{2}} \leq t_{n-1}^{\frac{k}{2} -1}$ and this implies that \[2t_{n-1}^{\frac{k}{2}} \leq t_{n-1}^{\frac{k}{2}}+t_{n-1}^{\frac{k}{2} -1} \leq M.\]
	This inequality gives that $\bB_{2^n}\left(2t_{n-1}^{\frac{k}{2}}\right) \subseteq \bB_{2^n}(M)$. 
	Since $\mu, \nu \in \mathcal{H}_{2^{n-1}}^{\frac{k}{2}} = \bB_{2^{n-1}}\left(t_{n-1}^{\frac{k}{2}}\right) $, we have $\lambda \in \bB_{2^{n-1}}\left(t_{n-1}^{\frac{k}{2}}\right) \Diamond \bB_{2^{n-1}}\left(t_{n-1}^{\frac{k}{2}}\right)$. Using Lemma \ref{lem: B-diamond} we conclude that
	\[\lambda \in \bB_{2^{n-1}}\left(t_{n-1}^{\frac{k}{2}}\right) \Diamond \bB_{2^{n-1}}\left(t_{n-1}^{\frac{k}{2}}\right) = \bB_{2^n}\left(2t_{n-1}^{\frac{k}{2}}\right) \subseteq \bB_{2^n}(M).\]
	This is a contradiction as $\lambda_1 \geq M+1$.
		
	\item Suppose instead that $\phi=(\phi_1 \times \phi_2) \up^{P_{2^n}}$, where $\phi_1, \phi_2 \in \Irr(P_{2^{n-1}})$ and $\phi_1 \neq \phi_2$. Then there exist $\mu_1, \mu_2 \in \mathcal{H}(2^{n-1})$ such that $\mathcal{LR}(\lambda; \mu_1, \mu_2)\neq 0$ and $\phi_1$ (respectively $\phi_2$) is an irreducible constituent of $\mu_1 \down_{P_{2^{n-1}}}$ (respectively $\mu_2 \down_{P_{2^{n-1}}}$).
	We need to distinguish two cases: the first one holds when $\phi_1(1)=2^i$ and $\phi_2 (1)=2^j$ , with $i,j \in \left[0, \alpha_{2^{n-1}} \right]$, $i\neq j$ and $i+j=k-1$. 
	In this case $\mu_1 \in \mathcal{H}_{2^{n-1}}^i$ and $\mu_2 \in \mathcal{H}_{2^{n-1}}^j$. Hence, by inductive hypothesis, Lemma \ref{lem: B-diamond} and the definition of $M$ we get that 
	\[\lambda\in \mathcal{H}_{2^{n-1}}^i \Diamond \mathcal{H}_{2^{n-1}}^j = \bB_{2^{n-1}}\left(t_{n-1}^i\right) \Diamond \bB_{2^{n-1}}\left(t_{n-1}^j\right) = \bB_{2^{n
	}}\left(t_{n-1}^i + t_{n-1}^j\right) \subseteq \bB_{2^n}(M).\]
This is a contradiction as $\lambda_1 \geq M+1$.
		
	\item The last case to consider is the one where $\phi_1 \neq \phi_2$ but $\phi_1(1)=\phi_2(1)=2^{w}$. In this setting we have that $\mu_1, \mu_2 \in \mathcal{H}_{2^{n-1}}^{w}$. By inductive hypothesis $\mathcal{H}_{2^{n-1}}^{w}=\bB_{2^{n-1}}\left(t_{n-1}^{w}\right)$. In particular, $(\mu_1)_1, (\mu_2)_1 \leq t_{n-1}^{w}$.
	Hence we have
		\begin{equation} \label{eq: thm_last case}
			2 t_{n-1}^{w} \leq M+1 \leq \lambda_1 \leq (\mu_1)_1 + (\mu_2)_1 \leq 2 t_{n-1}^{w},
		\end{equation}
		where the first equality holds by definition of $M$, the second one by assumption and the third one by Lemma \ref{lem: LR_prop}. Therefore (\ref{eq: thm_last case}) is a chain of equalities and in particular, $(\mu_1)_1 = (\mu_2)_1 = t_{n-1}^{w}$. Hence $\mu_1=\mu_2= \lambda_{n-1}^{w}$. Then $\delta_{n-1}^{w}=0$ since by assumption $\phi_1$ and $\phi_2$ are two distinct irreducible constituents of $\lambda_{n-1}^{w} \down_{P_{2^{n-1}}}$ of degree $2^{w}$.
	By definition, \[M=\max \Set{t_{n-1}^i + t_{n-1}^j, 2t_{n-1}^{w} | i,j,w \in \left[0, \alpha_{2^{n-1}} \right],\ i+j=k-1 \mbox{ and } i\neq j }.\] Since (\ref{eq: thm_last case}) is a chain of equalities we get $M+1=\lambda_1=2t_{n-1}^{w}\leq M$, which is  a contradiction. 
	\end{enumerate}
\end{proof}

\begin{theorem}\label{thm: generic_case}
		Let $n \in \mathbb{N}$ and $k\in [0, \alpha_{n}]$. Then:
	\begin{itemize}
		\item[(1)] there exists $T_n^k \in [1,n]$ such that $\mathcal{H}_{n}^k = \bB_{n}(T_n^k)$;
		\item[(2)] if $k>1$, for every $\lambda \in \bB_{n}(T_n^k -1)$, $\chi^\lambda \down_{P_{n}}$ has three distinct irreducible constituents of degree $2^k$. %necessarely $n \geq 8$, since $2^2$ has at most $k=1$ => at leat $2^3$
	\end{itemize}
\end{theorem}
\begin{proof}
	We proceed by induction on $n$: if $n=1$ then $k=0$ and the statement is obvious. 	Let $n\geq 2$ and let $n=\sum_{i=1}^r 2^{n_i}$ be its binary expansion. By Theorem \ref{thm: 2_power}, for every $i\in [1,r]$ and every $d_i \in \left[0, \alpha_{2^{n_i}}\right]$ there exists $t_{n_i}^{d_i} \in \left[2^{n_i -1}+1, 2^{n_i}\right]$ such that $\mathcal{H}_{2^{n_i}}^{d_i}=\bB_{2^{n_i}}\left(t_{n_i}^{d_i}\right)$. 
	%Fix $t_{n_i}^{d_i}=0$ whenever $n_i=0$.
	Define
	\[M:= \max \Set{t_{n_1}^{j_1} + \cdots + t_{n_r}^{j_r} | j_i \in \left[0, \alpha_{2^{n_i}}\right] \mbox{ for every } i \in [1,r] \mbox{ and } j_1 + \cdots + j_r =k }. \]
	We want to prove that $\mathcal{H}_n^k =\bB_n(M)$.
	Let $j_1 \in \left[0, \alpha_{2^{n_1}}\right], \dots , j_r \in \left[0, \alpha_{2^{n_r}}\right]$ be such that $M=t_{n_1}^{j_1} + \cdots + t_{n_r}^{j_r} $. We have 
	\[ \bB_n(M) = \bB_{2^{n_1}}(t_{n_1}^{j_1}) \Diamond \cdots \Diamond \bB_{2^{n_r}}(t_{n_r}^{j_r}) 
	= \mathcal{H}_{2^{n_1}}^{j_1} \Diamond \cdots \Diamond \mathcal{H}_{2^{n_r}}^{j_r}
	\subseteq \mathcal{H}_n^k ,\]
	where the first equality holds by Lemma \ref{lem: B-diamond}, the second by Theorem \ref{thm: 2_power} and the inclusion by Lemma \ref{lem: O_diamond_general}.
	
	To prove part (1) of the theorem it remains to show that $\mathcal{H}_n^k \subseteq \bB_n(M)$. Let $\lambda \in \mathcal{H}_n^k$. 
	%Suppose by contradiction that there exists $\lambda \in \mathcal{H}_n^k \setminus \bB_n(M)$, and without loss of generality that $\lambda_1 \geq M+1$.
	Then there exists an irreducible constituent $\phi$ of $\chi^\lambda \down_{P_{n}}$ of degree $2^k$. 
	The structure of $P_n$ discussed in section \ref{sec: wr} implies that $\phi=\phi_1 \times \cdots \times \phi_r$, for some $\phi_i \in \Irr(P_{2^{n_i}})$ for every $i\in [1,r]$ such that $\phi_i(1)=2^{j_i}$ and such that $j_1+ \cdots +j_r=k$. Hence there exists $\mu_i \in \mathcal{H}(2^{n_i})$ with $\phi_i$ as an irreducible constituent of the restriction $\chi^{\mu_i} \down_{P_{2^{n_i}}}$ for every $i\in [1,r]$, such that $\mathcal{LR}(\lambda; \mu_1, \dots , \mu_r)\neq 0$. In particular, by Theorem \ref{thm: 2_power} there exists $t_{n_i}^{j_i} \in \left[2^{n_i -1}+1, 2^{n_i}\right] $ such that  $\mu_i \in \mathcal{H}_{2^{n_i}}^{j_i} =\bB_{2^{n_i}}(t_{n_i}^{j_i})$. Therefore by Lemma \ref{lem: B-diamond},
	\[\lambda \in \bB_{2^{n_1}}\left(t_{n_1}^{j_1}\right) \Diamond \cdots \Diamond \bB_{2^{n_r}}\left(t_{n_r}^{j_r}\right)= \bB_n\left(t_{n_1}^{j_1}+ \cdots + t_{n_r}^{j_r} \right)\subseteq \bB_n\left(M \right),\]
where the last inclusion follows from the definition of $M$. 
	
In order to prove statement (2), let us fix $k>1$. From the discussion above, we know that $\mathcal{H}_{n}^k = \bB_{n}(T_n^k)$, where 
$T_n^k = t_{n_1}^{j_1}+ \cdots + t_{n_r}^{j_r} $ for suitable $j_1 \in [0, \alpha_{2^{n_1}}], \dots , j_r \in \left[0, \alpha_{2^{n_r}}\right]$ such that $j_1+\cdots +j_r=k$.
	Since $k>1$, we can suppose without loss of generality that $j_1>1$. %in particular $n_1>1$
	Let $\lambda \in \bB_n(T_n^k -1)$. By Lemma \ref{lem: B-diamond}, $\bB_n(T_n^k -1)=\bB_{2^{n_1}}\left(t_{n_1}^{j_1} -1 \right) \Diamond \cdots \Diamond \bB_{2^{n_r}}\left(t_{n_r}^{j_r}\right)$. Hence there exist $\mu_1 \in \bB_{2^{n_1}}\left(t_{n_1}^{j_1} -1 \right)$ and $\mu_i \in \bB_{2^{n_i}}\left(t_{n_i}^{j_i}\right)$ for every $i\in [2,r]$ such that $\mathcal{LR}(\lambda; \mu_1, \mu_2 , \dots , \mu_r)\neq 0$.
	By definition there exists an irreducible constituent $\psi_i$ of $\chi^{\mu_i} \down_{P_{2^{n_i}}}$ of degree $2^{j_i}$ for every $i\in [2,r]$. By Theorem \ref{thm: 2_power}, there exist three distinct irreducible constituents $\phi_1, \phi_2$ and $\phi_3$ of $\chi^{\mu_1} \down_{P_{2^{n_1}}}$ of degree $2^{j_1}$.
	Therefore $\phi_1 \times \psi_2 \times \cdots \times \psi_r$, $\phi_2 \times \psi_2 \times \cdots \times \psi_r$ and $\phi_3 \times \psi_2 \times \cdots \times \psi_r$ are three distinct irreducible constituents of $\chi^\lambda \down_{P_{n}}$ of degree $2^k$.
\end{proof}

We conclude the article by explicitly computing $T_n^{\alpha_n}$. By doing this, we manage to identify those characters $\chi\in\mathrm{Irr}_{\mathcal{H}}(\fS_n)$ such that $\chi\down_{P_n}$ admits irreducible constituents of every possible degree. 

\begin{definition}
For $n\in \mathbb{N}$, define $\tau_n \in [2^{n-1}+1, 2^n]$ by
\[ \begin{split}
&\tau_1=2,\ \tau_2=3, \tau_3=7,\ \tau_4=13,\ \tau_5=26, \mbox{ and}
	\\ &\tau_n= 2^{n-1}+2^{n-2}+2^{n-5}+2^{n-6} \mbox{ for } n\geq 6 . \end{split} \]
Notice that $\tau_n=2\tau_{n-1}$ for any $n\geq 7$.
\end{definition}

As usual, we start by studying the case where $n$ is a power of $2$. 

\begin{proposition} \label{prop: bound_2power}
	For every $n\in \mathbb{N}$, $T_{2^n}^{\alpha_{2^n}}=\tau_n$.
	Moreover, if $n\geq 6$ and $\lambda \in \mathcal{H}_{2^n}^{\alpha_{2^n}}$ then there exist at least three distinct irreducible constituents of $\chi^\lambda \down_{P_{2^n}}$ of degree $2^{\alpha_{2^n}}$.
\end{proposition}
\begin{proof}
We proceed by induction on $n$. If $n\leq 6$ then the statement holds by direct computation. 
	The cases $n=1,2,3$ can be seen in Example \ref{ex: n=1,2,3}, while the cases $n=4,5,6$ can be computed using Theorem \ref{thm: 2_power}. Notice that for $n\in [1,5]$, $\lambda =\left(t_n^{\alpha_{2^n}} , 1^{n-t_n^{\alpha_{2^n}}}\right) \in \mathcal{H}_{2^n}^{\alpha_{2^n}}$ does not have in $\chi^\lambda \down_{P_{2^n}}$ three distinct irreducible constituents of degree $2^{\alpha_{2^n}}$. Instead, it is true for $n=6$.

Let $n\geq 7$. We want to show that $\mathcal{H}_{2^n}^{\alpha_{2^n}}=\bB_{2^n}(\tau_{n})$.
	Given $\lambda \in \mathcal{H}_{2^n}^{\alpha_{2^n}}$, there exists an irreducible constituent $\phi $ of $\chi^\lambda \down_{P_{n}}$ of degree $2^{\alpha_{2^n}}$. As we have seen in the proof of Proposition \ref{prop: cd_n}, $\phi= (\phi_1 \times \phi_2)\up^{P_{2^n}}$ with $\phi_1,\phi_2 \in \Irr(P_{2^{n-1}})$, $\phi_1 \neq \phi_2$ and $\phi_1(1)=\phi_2(1)=2^{\alpha_{2^{n-1}}}$. Hence, there exists $\mu_1, \mu_2\in\mathcal{H}(2^{n-1})$ such that $\mathcal{LR}(\lambda; \mu_1, \mu_2)\neq 0$ and such that $\left[\chi^{\mu_i}\down_{P_{2^{n-1}}},\phi_i\right]\neq 0$ for all $i\in\{1,2\}$. In particular, $\mu_1,\mu_2\in \mathcal{H}_{2^{n-1}}^{\alpha_{2^{n-1}}}$ and by inductive hypothesis we know that $\mathcal{H}_{2^{n-1}}^{\alpha_{2^{n-1}}}=\bB_{2^{n-1}}(\tau_{n-1})$. Using these observations together with Lemma \ref{lem: B-diamond} we conclude that $$\lambda \in \bB_{2^{n-1}}(\tau_{n-1}) \Diamond \bB_{2^{n-1}}(\tau_{n-1})=\bB_{2^{n}}(2\tau_{n-1}) =\bB_{2^{n}}(\tau_{n}).$$ 
	%Hence $\mathcal{H}_{2^n}^{\alpha_{2^n}}\subseteq \bB_{2^n}(\tau_n)$.	
	
	In order to prove the other inclusion, we consider $\lambda \in \bB_{2^{n}}(\tau_{n})$. From  Lemma \ref{lem: B-diamond} and the inductive hypothesis, we know that  $$\bB_{2^{n}}(\tau_{n})= \bB_{2^{n-1}}(\tau_{n-1}) \Diamond \bB_{2^{n-1}}(\tau_{n-1}) = \mathcal{H}_{2^{n-1}}^{\alpha_{2^{n-1}}} \Diamond \mathcal{H}_{2^{n-1}}^{\alpha_{2^{n-1}}} .$$ Therefore there exist $\mu, \nu \in \mathcal{H}_{2^{n-1}}^{\alpha_{2^{n-1}}}$ such that $\chi^\mu \times \chi^\nu \mid (\chi^\lambda)\down_{\fS_{2^{n-1}} \times \fS_{2^{n-1}}}$. The inductive hypothesis implies that both $\chi^\mu \down_{P_{2^{n-1}}}$ and $\chi^\nu \down_{P_{2^{n-1}}}$ admit three distinct irreducible constituents of degree $2^{\alpha_{2^{n-1}}}$. Denote by $\phi_1, \phi_2, \phi_3$ those constituents of $\chi^\mu \down_{P_{2^{n-1}}}$ and by $\psi_1, \psi_2, \psi_3$  those constituents of $\chi^\nu \down_{P_{2^{n-1}}}$. If $\psi_j\notin\{\phi_1,\phi_2,\phi_3\}$ for some $j\in[1,3]$, then $(\phi_1 \times \psi_j) \up^{P_{2^n}}, (\phi_2 \times \psi_j) \up^{P_{2^n}}$ and $(\phi_3 \times \psi_j) \up^{P_{2^n}}$ are three distinct irreducible constituents of $\chi^\lambda \down_{P_{2^n}}$ of degree $2^{\alpha_{2^n}}$. On the other hand, if $\psi_j\in\{\phi_1,\phi_2,\phi_3\}$ for all $j\in[1,3]$, then we can assume without loss of generality that $\phi_i=\psi_i$ for all $i\in [1,3]$. In this case we have that $(\phi_1 \times \psi_2) \up^{P_{2^n}}, (\phi_2 \times \psi_3) \up^{P_{2^n}}$ and $(\phi_3 \times \psi_1) \up^{P_{2^n}}$ are three distinct irreducible constituents of $\chi^\lambda \down_{P_{2^n}}$ of degree $2^{\alpha_{2^n}}$. Moreover, $\lambda \in \mathcal{H}_{2^n}^{\alpha_{2^n}}$ as desired. 
\end{proof}

\begin{proposition}\label{prop: bound_n}
	If $n\in \mathbb{N}$ and $n=\sum_{i=1}^r 2^{n_i}$ is its binary expansion, then $T_n^{\alpha_n}=\sum_{i=1}^r \tau_{n_i}$.
\end{proposition}
\begin{proof}
	Arguing exactly as in the proof of Theorem \ref{thm: generic_case}, we have that
	\[T_n^{\alpha_n} = \max \lbrace T_{n_1}^{j_1}+\cdots +T_{n_r}^{j_r} \mid j_i\in [0,\alpha_{2^{n_i}}] \mbox{ for } i\in [1,r], \mbox{ and } \sum_{i=1}^r j_i = \alpha_n\} .\] 
Since $\alpha_n=\sum_{i=1}^r\alpha_{2^{n_i}}$, we deduce that $T_n^{\alpha_n}=T_{n_1}^{\alpha_{2^{n_1}}} +\cdots +T_{n_r}^{\alpha_{2^{n_r}}}$. 
The statement now follows from Proposition \ref*{prop: bound_2power}.
\end{proof}

\begin{observation}
	By Theorems \ref{thm: inclusion} and \ref{thm: generic_case} we know that for every $k\in [0, \alpha_n]$ we have
	$\bB_n(T_n^{\alpha_n})=\mathcal{H}_n^{\alpha_n} \subseteq \mathcal{H}_n^k$. 
From Proposition \ref{prop: bound_n} we observe that the majority of the elements of $\mathcal{H}(n)$ are contained in $\bB_n(T_n^{\alpha_n})$. This shows that the restriction to $P_n$ of most of the irreducible characters labelled by hook partitions admit irreducible constituents of every possible degree. 
	%	
	%Moreover, $T_n^{\alpha_n}$ is exactly the number of hook partitions $\lambda$ of $n$ such that in $\chi^\lambda \down_{P_n}$ have at least an irreducible constituent for every possible degree $2^0,2^1, \dots , 2^{\alpha_n}$.
	%If we relate this number to the hook partitions with first row longer than the first column, i.e. $|\mathcal{H}(2^n)|/2=2^{n-1}$, we obtain $\frac{1}{2}+\frac{1}{16}+\frac{1}{32} > \frac{1}{2}$. So more than an half of the considered hook partitions see every possible degree.
	%? è vero? Non dovremmo dividere per $2^n$, cioè tutti gli hook?
\end{observation}

\end{document}